\documentclass[a4paper,reqno]{amsart}

\usepackage[maxbibnames=9]{biblatex}

\usepackage{color}
\usepackage{graphicx}
\usepackage[normalem]{ulem}
\usepackage{morefloats}
\usepackage{hyperref}
\numberwithin{equation}{section}
\DeclareNameAlias{default}{last-first}
\hypersetup{
    bookmarks=true,         
    unicode=false,          
    pdftoolbar=true,        
    pdfmenubar=true,        
    pdffitwindow=false,     
    pdfstartview={FitH},    
    pdftitle={My title},    
    pdfauthor={Author},     
    pdfsubject={Subject},   
    pdfcreator={Creator},   
    pdfproducer={Producer}, 
    pdfkeywords={keyword1, key2, key3}, 
    pdfnewwindow=true,      
    colorlinks=false,       
    linkcolor=green,          
    citecolor=green,        
    filecolor=green,      
    urlcolor=green,           
    urlbordercolor={1 1 1}  
}

\usepackage{amsmath, amsfonts, amsthm, amssymb}
\usepackage{fancyhdr}

\usepackage{biblatex}

\usepackage{mathrsfs}
\usepackage{blkarray}
\usepackage{amsthm}
\usepackage{mathtools}
\usepackage[title]{appendix}

\newtheorem{theorem}{Theorem}[section]
\newtheorem*{theorem*}{Theorem}
\newtheorem{lemma}[theorem]{Lemma}

\newcommand{\DD}{\mathbb{D}}

\newtheorem{proposition}[theorem]{Proposition}

\theoremstyle{definition}
\newtheorem{definition}[theorem]{Definition}

\newtheorem{remark}[theorem]{Remark}

\theoremstyle{definition}
 
\theoremstyle{definition}
\newtheorem{question}[theorem]{Question}

\begin{document}

\title[\resizebox{2.9in}{!}{  \large POSITIVITY CONDITIONS ON THE ANNULUS}]{\large POSITIVITY CONDITIONS ON THE ANNULUS VIA THE DOUBLE-LAYER POTENTIAL KERNEL}

\author[\small M. T.  JURY]{MICHAEL T. JURY}
\address{DEPARTMENT OF MATHEMATICS, UNIVERSITY OF FLORIDA, GAINESVILLE, FL}
\email{mjury@ufl.edu} 
\author[\small G. TSIKALAS]{GEORGIOS TSIKALAS}
\address{DEPARTMENT OF MATHEMATICS AND STATISTICS, WASHINGTON UNIVERSITY IN ST. LOUIS, ST. LOUIS, MO}
\email{gtsikalas@wustl.edu} 

\small
\begin{abstract}
    \small
We introduce and study a scale of operator classes on the annulus that is motivated by the $\mathcal{C}_{\rho}$ classes of $\rho$-contractions of Nagy and Foia\c{s}. In particular, our classes are defined in terms of the contractivity of the double-layer potential integral operator over the annulus. We prove that if, in addition, complete contractivity is assumed, then one obtains a complete characterization involving certain variants of the $\mathcal{C}_{\rho}$ classes. Recent work of Crouzeix-Greenbaum and Schwenninger-de Vries allows us to also obtain relevant K-spectral estimates, generalizing and improving existing results from the literature on the annulus. Finally, we exhibit a special case where these estimates can be significantly strengthened. 
\end{abstract}
\maketitle

\small
 \section{INTRODUCTION} 
 \large 
Let $\rho>0$. Denote by $\mathcal{C}_{\rho}$ the class of all bounded Hilbert space operators $T\in\mathcal{B}(H)$ that have a unitary $\rho$-dilation, i.e. there exists a Hilbert space $K\supset H$ and a unitary $U\in\mathcal{B}(K)$ such that 
$$T^n=\rho P_H U^n|_H, \hspace{0.3 cm} n=1, 2,\dots,$$
where $P_H$ is the orthogonal projection of $K$ onto $H$. The elements of $\mathcal{C}_{\rho}$ are referred to as $\rho$-contractions. $\mathcal{C}_{\rho}$ was introduced by Sz.-Nagy and Foia\c{s} \cite{NagyFoiasoncertainclasses} (see also \cite[Chapter 1]{NagyFoiasbook}) and has subsequently been investigated by many authors (see e.g. \cite{Harnackcrho} and the references therein). It is known that $\mathcal{C}_1$ is precisely the set of contractions on $H$ \cite{NagyFrench}, while $\mathcal{C}_2$ is
the set of operators whose numerical range $W(T)=\{\langle Tx, x\rangle\text{ } | \text{ } ||x||=1 \}$ is contained in $\overline{\mathbb{D}}=\{|z|\le 1\}$ \cite{Bergerstampflimapping}, \cite{Bergerthesis}. \par
Now, for any open, bounded domain $\Omega\subset\mathbb{C}$, let $\mathcal{A}(\Omega)$ denote the uniform algebra of continuous functions $f$ on $\overline{\Omega}$ that are holomorphic on $\Omega$. Also, let  $T\in\mathcal{B}(H)$ be such that $\sigma(T)\subset\mathbb{D}.$ An alternate characterization of $\mathcal{C}_{\rho}$ (see \cite[p. 315]{CassierFackContractions}) then states that $T$ is a $\rho$-contraction if and only if the operator 
$$S_{\mathbb{D}, \rho}: \mathcal{A}(\mathbb{D})\to \mathcal{B}(H) $$
\begin{equation} \label{2001}  f\mapsto \frac{1}{\rho}\bigg[f(T)+(C\overline{f})(T)^*+\frac{(\rho-2)}{2\pi i}\int_{\partial\mathbb{D}}f(\zeta)d\zeta\bigg]\end{equation}
$$=\frac{1}{\rho}[f(T)+(\rho-1)f(0)] $$
is contractive, where $$(C\overline{f})(z)=\frac{1}{2\pi i}\int_{\partial\mathbb{D}}\frac{\overline{f(\zeta)}}{\zeta-z}d\zeta \hspace{0.3 cm} (z\in\mathbb{D})$$
denotes the Cauchy transform of $\overline{f}.$  The mapping $S_{\mathbb{D}, \rho}$ has proved itself a valuable tool for better understanding $\mathcal{C}_{\rho}$. As a recent example, Clouâtre, Ostermann and Ransford  \cite{RansfordOster}, building on ideas from \cite{LiCaldwellGreen}, used the contractivity of $S_{\mathbb{D}, \rho}$ as a stepping stone for an alternative, simple proof of the fact that, for any $\rho\ge 1$ and $T\in\mathcal{C}_{\rho},$ 
$$||p(T)||\le \rho \sup_{z\in\overline{\mathbb{D}}} |p(z)|, \hspace{0.3 cm} \forall p\in C[z].$$
In other words, $\overline{\mathbb{D}}$ is a $\rho$-spectral set for $T$ whenever $T$ is a $\rho$-contraction (with $\rho\ge 1$), a result originally proved by Okubo and Ando via dilation theory \cite{OkuboAndroconstantsrelated}. \par More generally, the unit disk $\mathbb{D}$ can be replaced by any (smoothly bounded) open convex set $\Omega\subset\mathbb{C}.$ In this setting, the contractivity of the \textit{double-layer potential} integral operator 
\[ S_{\Omega, 2}: \   f\mapsto \frac{1}{2}\big[f(T)+(C\overline{f})(T)^*\big]\]
(which is equivalent to the inclusion $\overline{W(T)}\subset \Omega$) has served as a basis for establishing $K$-spectral estimates for $W(T)$; an approach originally due to Delyon and Delyon \cite{Delyon}, it was further refined by Crouzeix in \cite{CrouxNumerangeandfunctional} and culminated in the Crouzeix-Palencia paper \cite{CrouzPal} (see also \cite{RansfordSchwenninger}, \cite{manyauthorscrouz} and the recent preprint \cite{schwendevries} for further developments), where it was shown that $\overline{W(T)}$ is always a $(1+\sqrt{2})$-spectral set for any $T\in\mathcal{B}(H).$   \par 
More recently, the study of $K$-spectral estimates through the contractivity of $S_{\Omega, 2}$ has expanded beyond convex domains (see \cite{LiCaldwellGreen}, \cite{CrouzGreen}). One important example in that direction, and what actually served as the main motivation for this paper, revolves around the annulus $A_R=\{1/R<|z|<R\}$ and the associated operator class $\mathbb{QA}_R$, termed the \textit{quantum annulus}, that consists of all invertible Hilbert space operators $T$ satisfying $||T||, ||T^{-1}||\le R$ (for recent results regarding $\mathbb{QA}_R$, see also  \cite{belloyak},\cite{mcculloughpascoe}, \cite{Tsiknote} and \cite{Tsikvonneumann}). In particular, in \cite[Section 5]{CrouzGreen} it was shown that the mapping 
$$ S_{R, 0}: \mathcal{A}(A_R)\to \mathcal{B}(H) $$
\begin{equation*}  f\mapsto \frac{1}{2}\big[f(T)+(C\overline{f})(T)^*\big], \end{equation*}
 where $(C\overline{f})(z)=\frac{1}{2\pi i}\int_{A_R}\frac{\overline{f(\zeta)}}{\zeta-z}d\zeta$, will always be contractive if $T\in \mathbb{QA}_R$. This observation, combined with the contractivity of $f\mapsto C\overline{f}$ over $A_R$ and an abstract functional analysis lemma (see \cite[Theorem 2]{CrouzGreen} or \cite[Lemma 1.1]{RansfordSchwenninger}), suffices to establish that, given any $T\in\mathbb{QA}_R$, we must have
 $$||f(T)||\le (1+\sqrt{2})\sup_{z\in\overline{A_R}}|f(z)|, \hspace{0.3 cm} \forall f\in  \mathcal{A}(A_R),$$
 i.e.  $\overline{A_R}$ will be a $(1+\sqrt{2})$-spectral set for $T$ whenever $T\in\mathbb{QA}_R$ (it is known that the optimal value of this spectral constant cannot be less than $2$, see \cite{Tsiknote}). \par 
Now, observe that in the previous discussion, the class $\mathbb{QA}_R$ enters the picture only through the contractivity of $S_{R, 0}$. Thus, if one wants to gain a better understanding of $K$-spectral estimates over $A_R$, the following question emerges naturally: for which operators $T$ will $S_{R, 0}$ be contractive? This line of inquiry, together with the form of the mapping (\ref{2001}), is what motivated our definition of the operator class $\mathbb{DLA}_R(c)$ (where $c>-2$);  an operator $T\in\mathcal{B}(H)$ with $\sigma(T)\subset A_R$ will belong to $\mathbb{DLA}_R(c)$ if and only if the mapping 
$$ S_{R, c}: \mathcal{A}(A_R)\to \mathcal{B}(H) $$
\begin{equation*}  f=\sum_{n=-\infty}^{\infty}a_nz^n\mapsto \frac{1}{2+c}\big[f(T)+(C\overline{f})(T)^*+ca_0\big] \end{equation*}
is contractive (see Section \ref{basic} for the general definition and Theorem \ref{doublelayer} for the equivalence between the two when $\sigma(T)\subset A_R$). The goal of the present work is to study the operator class $\mathbb{DLA}_R(c)$ and its completely contractive analogue $\mathbb{CDLA}_R(c)$ (see Section \ref{completeversion}). 
Note that the study of completely bounded maps and dilations in the setting of the double-layer potential kernel and Crouzeix's conjecture has already been successfully initiated in the papers \cite{Putsand} and \cite[Section 6]{RansfordOster}. Also, we point out that, while the inclusion $\mathbb{CDLA}_R(c)\subseteq \mathbb{DLA}_R(c)$ follows immediately from the definitions, it is not known to us whether it is strict or not (see Question \ref{quest}).
\par 
To state our first main result, a characterization of $\mathbb{CDLA}_R(c)$, we require the following generalization of the $\mathcal{C}_{\rho}$ classes, introduced by Langer \cite[p.  53]{NagyFoiasbook} (see also \cite{WA}).
\begin{definition}\label{CAAAA}
Assume $A\in\mathcal{B}(H)$ is a bounded, positive operator that is also bounded below. The class $\mathcal{C}_A$ contains all operators $T\in\mathcal{B}(H)$ with the property that there exists a Hilbert space $K\supset H$ and a unitary $U\in\mathcal{B}(K)$ such that 
$$A^{-1/2}T^nA^{-1/2}=P_HU^n|_H, \hspace{0.3 cm} n=1, 2, \dots.$$
\end{definition}
Our characterization then proceeds as follows. Note that the inequality $T\ge 0$ implies that the Hilbert space operator $T$ is positive, while $T>0$ implies that it is strictly positive. 
\begin{theorem} \label{2002}
Let $T\in \mathcal{B}(H)$ and $c>-2$, with $\sigma(T)\subset A_R$. Then, $T\in \mathbb{CDLA}_R(c)$ if and only if $T/R\in \mathcal{C}_{2+c-A}$ and $T^{-1}/R\in \mathcal{C}_A $ for some $A\in\mathcal{B}(H)$ such that $A>0$ and $2+c-A>0$. 
\end{theorem}
Dropping the assumption $\sigma(T)\subset A_R$ leads us to Theorem \ref{droptheorem}. \par 
As remarked previously, the only way the class $\mathbb{QA}_R$ enters in the proof of the $K=1+\sqrt{2}$ spectral estimate in \cite{CrouzGreen} is through the inclusion $\mathbb{QA}_R\subset \mathbb{CDLA}(0)$. This suggests that the spectral constant for $\mathbb{QA}_R$ may coincide with the one for $\mathbb{CDLA}_R(0)$. Utilizing the solution of an extremal problem over $A_R$ due to McCullough and Shen \cite{Mccshen}, we are able to prove a partial result in the setting of $2\times 2$ matrices that supports this idea. 
\begin{theorem}\label{2x2}
Let $c\ge 0$ and assume $T\in \mathbb{DLA}_{R}(c)$ is a $2\times 2$ matrix with a single eigenvalue. Then, $\overline{A_R}$ will be a $K(R)$-spectral set for $T$, where 
$$K(R)=2+c\frac{R^2-1}{R^2+1}\le 2+c.$$
\end{theorem}
Note that one can also take advantage of the machinery established in \cite{LiCaldwellGreen} and \cite{CrouzGreen} to prove general $K$-spectral estimates for $\mathbb{DLA}_R(c)$ and $\mathbb{CDLA}_R(c)$; see Theorem \ref{generalspectral} (see also Remark \ref{devriesremark}). In fact, our approach yields sharper estimates for certain operator classes; see Remark \ref{improvenum}. \par 
Our paper is organized as follows: Section \ref{prelims} contains a few preliminary lemmata on the $\mathcal{C}_A$ classes. In Section \ref{basic}, we explore basic properties of $\mathbb{DLA}_{R}(c)$, mostly related to certain inclusion and monotonicity results. In Section \ref{completeversion}, we characterize $\mathbb{CDLA}_{R}(c)$ through Theorems \ref{2002} and \ref{droptheorem}. Finally, Section \ref{kspectral} contains the proofs of Theorems \ref{2x2} and \ref{generalspectral}.
 
 \small
\section{PRELIMINARIES} \label{prelims}
\large

\subsection{The $\mathcal{C}_A$ classes} 
This section contains  alternate characterizations of $\mathcal{C}_A$ that are usually easier to work with. We have included proofs for the convenience of the reader. \par 

\begin{lemma} \label{CAchar}
Assume $A\in\mathcal{B}(H)$ is a bounded, strictly positive operator and let $T\in\mathcal{B}(H)$ . Then, $T\in\mathcal{C}_A$ if and only if $\sigma(T)\subset\overline{\mathbb{D}}$ and
\begin{equation} \label{CAdescr}
2\Re(1-zT)^{-1}+A-2\ge 0, \hspace{0.3 cm}  \forall z\in \mathbb{D}.
\end{equation}
\end{lemma}
\begin{proof}
Assume first that $T\in\mathcal{C}_A$. Since $\mathcal{C}_A\subset \mathcal{C}_{||A||},$ we must have $\sigma(T)\subset\overline{\mathbb{D}}$ (see e.g. \cite[p. 43]{NagyFoiasbook}). Also, in view of  \cite[p.  53]{NagyFoiasbook}, we can deduce that 
$$ 
\langle Ah, h\rangle -2\Re\langle z(A-I)Th, h\rangle +|z|^2\langle (A-2I)Th, Th\rangle \ge 0, 
$$ 
for all $h\in H, |z|\le 1.$ This inequality can be rewritten as 
 $$ 
\langle A(I-zT)h, (I-zT)h\rangle -2\langle (I-zT)h, (I-zT)h\rangle +2\Re \langle (I-zT)h, h\rangle \ge 0
$$  
for all $h\in H, |z|\le 1.$ Setting $h=(I-zT)^{-1}h,$ we obtain
$$\langle (A-2+(I-zT)^{-1}+(1-\overline{z}T^*)^{-1})h, h\rangle \ge 0 $$
for all $h\in H, |z|\le 1,$ which is the desired conclusion. \par 
For the converse, simply roll back the steps in the previous proof. 
\end{proof}
Note that one could define $\mathcal{C}_A$ more generally, for $A$ bounded and self-adjoint, through (\ref{CAdescr}). Using this definition, it is easy to see that $\mathcal{C}_A$ is non-empty if and only if $A\ge 0$ (in which case it will contain the zero operator). 
\begin{lemma} 
Assume $A\in\mathcal{B}(H)$ is a bounded, strictly positive operator and let $T\in\mathcal{B}(H)$ be such that $\sigma(T)\subset \mathbb{D}$ . Then, $T\in\mathcal{C}_A$ if and only if 
\begin{equation} \label{fromztotheta}
(1-e^{i\theta}T)^{-1}+(1-e^{-i\theta}T^*)^{-1}+A-2\ge 0, \hspace{0.3 cm}  \forall \theta\in [0. 2\pi).
\end{equation}              
\end{lemma}
 \begin{proof}
Fix an arbitrary $h\in H$ and define 
$$\Phi_h: \mathbb{D}\to \mathbb{C} $$
$$z\mapsto \big\langle \big(2\Re(1-zT)^{-1}+A-2\big)h, h\big\rangle.$$
Since $\sigma(T)\subset \mathbb{D}$, $\Phi_h$ will be a harmonic function on $\mathbb{D}$ that extends continuously to $\overline{\mathbb{D}}$. By the minimum principle for harmonic functions, we then obtain that $\Phi_h(z)\ge 0,$ for all $z\in\DD,$ if and only if  
 $$\Phi_h(e^{i\theta})=\big\langle \big(2\Re(1-e^{i\theta}T)^{-1}+A-2\big)h, h\big\rangle\ge 0, \hspace{0.3 cm}  \forall \theta\in [0. 2\pi).$$
 Since $h$ was arbitrary, we are done.
 \end{proof}

\par

\subsection{Completely bounded maps} 
  
Let $A\subset \mathcal{B}(H)$ denote an \textit{operator algebra}, i.e. a unital subalgebra of the $C^*$-algebra of bounded linear operators on some Hilbert space $H$. Given a natural number $n \ge 1$, we denote by $M_n(A)$ the algebra of $n\times n$ matrices with entries from $A$, which we view as a subalgebra of bounded linear operators acting on $H^{(n)}=H\oplus H\oplus \dots\oplus H$. In particular, $M_n(A)$ is endowed with a norm under this identification. Given a map $\Phi: A\to\mathcal{B}(H)$, for each $n\ge 1$, we may define the coordinate-wise map $\Phi^{(n)}: M_n(A)\to\mathcal{B}(K^{(n)})$ as
$$\Phi^{(n)}([a_{ij}])=[\Phi(a_{ij})], \hspace{0.3 cm} [a_{ij}]\in M_n(A).$$
 If $\Phi$ is linear (or anti-linear), we say that $\Phi$ is completely bounded if the
quantity $$||\Phi||_{cb}=\sup_n||\Phi^{(n)}|| $$
is finite. We say that $\Phi$ is completely contractive if $||\Phi||_{cb}\le 1.$ Furthermore, $\Phi$ will be called positive if it maps positive elements of $A$ to positive operators in $\mathcal{B}(K).$ $\Phi$ will be said to be completely positive if $\Phi^{(n)}$ is positive for every $n\ge 1.$ For more details on these concepts, see \cite{Paulsenbook}.

 \small 
\section{THE $\mathbb{DLA}_{R}(c)$ CLASS}\label{basic}
 \large 

 \par 
\subsection{The $\mathcal{C}_{s, t}(R)$ classes}
 
Before we dive into the study of $\mathbb{DLA}_{R}(c)$, we will establish a few basic facts regarding $\mathcal{C}_{s, t}(R)$. These are operator classes which form (as will be seen in the next subsection) strict subsets of certain  $\mathbb{DLA}_{R}(c)$ classes.
\begin{definition}
Let $R>1$ and $s, t>0$. Define 
$$\mathcal{C}_{s, t}(R)=\{T\in\mathcal{B}(H)\text{ }|\text{ }T/R\in\mathcal{C}_s \text{ and } T^{-1}/R\in\mathcal{C}_t\}.$$
\end{definition}
First, we record $\mathcal{C}_{s, t}(R)$-membership criteria concerning normal matrices. These are easy consequences of known results about $\mathcal{C}_{\rho}$. In particular, we will be needing the following ``reciprocity" law, as seen in \cite{AndoNishioconvex}. 
\begin{lemma}\label{50}
 Assume $\rho\in (0, 2)$ and $T\in\mathcal{B}(H).$ Then, $T\in\mathcal{C}_{\rho}$ if and only if  $\frac{2-\rho}{\rho}T\in\mathcal{C}_{2-\rho}$.  
\end{lemma}

\begin{proposition}\label{51}
Fix $R>1$ and assume $s, t>0.$ Also, let $N\in\mathcal{B}(H)$ be normal.

\begin{itemize}
    \item[(i)] If $s, t\ge 1,$ then $\mathcal{C}_{s, t}(R)$ will always be non-empty. Moreover, $N\in \mathcal{C}_{s, t}(R)$ if and only if $\frac{1}{R^2}\le  N^*N \le R^2.$

    \item[(ii)] If $s\ge 1$ and $t\le 1,$ then $\mathcal{C}_{s, t}(R)$ will be non-empty if and only if   $$ \frac{2}{R^2+1}\le t.$$
    Moreover, $N\in \mathcal{C}_{s, t}(R)$ if and only if
        $\frac{(2-t)^2}{R^2t^2}\le N^*N\le R^2.$
    
    \item[(iii)] If $s\le 1$ and $t\ge 1,$ then $\mathcal{C}_{s, t}(R)$ will be non-empty if and only if $$ \frac{2}{R^2+1}\le s.$$
    Moreover, $N\in \mathcal{C}_{s, t}(R)$ if and only if
        $\frac{1}{R^2}\le N^*N\le R^2\frac{s^2}{(2-s)^2}.$
    
         \item[(iv)] If $s\le 1$ and $t\le 1,$ then $\mathcal{C}_{s, t}(R)$ will be non-empty if and only if $$\frac{1}{R}\frac{2-t}{t}\le R\frac{s}{2-s}.$$ Moreover, $N\in\mathcal{C}_{s, t}(R)$ if and only if
        $\frac{1}{R^2}\frac{(2-t)^2}{t^2}\le N^*N \le R^2\frac{s^2}{(2-s)^2}.$
    \end{itemize}
\end{proposition}
\begin{proof}
We only consider the case $s\ge 1,\ t\le 1$ (the remaining three cases can be proved in essentially identical ways). Since $N$ is normal and $s\ge 1,$ \cite[Lemma 4]{BergStampskewdil} tells us that $N/R\in\mathcal{C}_s$ if and only if $||N||\le R.$ Also, in view of Lemma \ref{50}, we can deduce that $N^{-1}/R\in\mathcal{C}_t$ if and only if 
$\frac{2-t}{Rt}N^{-1}\in\mathcal{C}_{2-t}$, which is equivalent (since $2-t\ge 1$) to the inequality $\frac{(2-t)^2}{R^2t^2}\le N^*N$. \par This proves the second half of (ii). To show the first statement, note that, since $ \frac{2}{R^2+1}\le t$, there exist scalars $a\in\mathbb{C}$ satisfying  $\frac{2-t}{Rt}\le |a| \le R$. In view of our previous result, these scalars must belong to $\mathcal{C}_{s ,t}(R),$ hence the class will be non-empty. Conversely, assume that $\frac{2}{R^2+1}> t$ and that there exists $T\in\mathcal{C}_{s, t}(R)$. Since $T/R\in\mathcal{C}_s$, we obtain $\sigma(T/R)\subset \{|z|\le 1\},$ hence $\sigma(T)\subset \{|z|\le R\}.$ Also, the fact that  $T^{-1}/R\in\mathcal{C}_t$ implies (in view of Lemma \ref{50}) that $(2-t)t^{-1}T^{-1}/R\in\mathcal{C}_{2-t}$, hence $\sigma((2-t)t^{-1}T^{-1}/R)\subset \{|z|\le 1\}.$ From that we can conclude
$$\sigma(T)\subset\{R^{-1}(2-t)t^{-1}\le |z|\le R \}=\emptyset,$$
as $\frac{2}{R^2+1}> t$, a contradiction. Thus, $\mathcal{C}_{s, t}(R)$ must be empty.
\end{proof}
We now show that the $\mathcal{C}_{s, t}(R)$ classes are, in a certain sense, ``rigid" with respect to the parameters $s ,t$. To do this, we require $\mathcal{C}_{s, t}(R)$-membership conditions for $2\times 2$ matrices with a single eigenvalue.

\begin{lemma} \label{90}
    Fix $R>1.$ Also, let $s, t>0$ and $T=\begin{pmatrix} a & b \\ 0 & a \end{pmatrix}
     $, where $a, b\in\mathbb{C}.$
     \begin{itemize}
         \item[(i)] Assume $s, t\ge 1.$  Then, $T\in\mathcal{C}_{s, t}(R)$ if and only if 
$\frac{1}{R}\le |a|\le R$
and $$R|b|-sR^2+(2-s)|a|^2\le 2(1-s)R|a|$$
and
$$R|b|-t|a|^2R^2+(2-t)\le 2(1-t)R|a|.$$

        \item[(ii)] Assume $s\ge 1$ and $t\le 1.$ Then, $T\in\mathcal{C}_{s, t}(R)$ if and only if 
$\frac{1}{R}\frac{2-t}{t}\le |a|\le R$ 
and $$R|b|-sR^2+(2-s)|a|^2\le 2(1-s)R|a|$$
and
$$R|b|-t|a|^2R^2+(2-t)\le 2(t-1)R|a|.$$
        \item[(iii)] Assume $s\le 1$ and $t\ge 1.$ Then, $T\in\mathcal{C}_{s, t}(R)$ if and only if 
$\frac{1}{R}\le |a|\le R\frac{s}{2-s}$ 
and $$R|b|-sR^2+(2-s)|a|^2\le 2(s-1)R|a|$$
and
$$R|b|-t|a|^2R^2+(2-t)\le 2(1-t)R|a|.$$
                 \item[(iv)] Assume $s\le 1$ and $t\le 1.$
                 Then, $T\in\mathcal{C}_{s, t}(R)$ if and only if 
$\frac{1}{R}\frac{2-t}{t}\le |a|\le R\frac{s}{2-s}$
and $$R|b|-sR^2+(2-s)|a|^2\le 2(s-1)R|a|$$
and
$$R|b|-t|a|^2R^2+(2-t)\le 2(t-1)R|a|.$$

     \end{itemize}
 \end{lemma}
 \begin{proof}
We can either verify the positivity conditions directly or simply invoke \cite[Theorem 3.1]{OkuboSpitkovsky2x2}.
 \end{proof}

\begin{theorem} \label{110}
Assume $s,t>0$ and $R>1$ are such that $\mathcal{C}_{s, t}(R)$ is non-empty. Then, there exists $T\in \mathcal{C}_{s, t}(R)$ such that for every $\epsilon>0$ and every $s', t'>0$ we have 
$$T\notin\mathcal{C}_{s', t-\epsilon}(R) \hspace{0.3 cm} \text{ and } \hspace{0.3 cm} T\notin\mathcal{C}_{s-\epsilon, t'}(R).$$
\end{theorem}
\begin{proof}
We only deal with the case $s\ge 1, t\le 1$, as the computations required for the remaining three cases are very similar in nature. \par 
So, let $s\ge 1, t\le 1$ and assume also that $(2-t)/(Rt)\le R,$ hence $\mathcal{C}_{s, t}(R)$ is non-empty. Note that if $(2-t)/(Rt)=R$, then any class either of the form $\mathcal{C}_{s', t-\epsilon}(R)$ or of the form $\mathcal{C}_{s-\epsilon, t'}(R)$ will be empty. Obviously, the conclusion of the theorem will hold in this case. Thus, we may actually assume that $(2-t)/(Rt)<R.$ We are looking for an operator that satisfies the conditions in the statement of the theorem and has the form  $T=\begin{pmatrix} a & b \\ 0 & a \end{pmatrix}
     $, where $a, b\in\mathbb{C}.$ In view of Lemma (\ref{90}), we know that such a  $T$ will belong to $\mathcal{C}_{s, t}(R)$ if and only if 
$\frac{1}{R}\frac{2-t}{t}\le |a|\le R$
and \begin{equation}\label{95}
 R|b|-sR^2+(2-s)|a|^2\le 2(1-s)R|a|   
\end{equation}
and
\begin{equation} \label{96}
 R|b|-t|a|^2R^2+(2-t)\le 2(t-1)R|a|.   
\end{equation}
Our goal is to find a pair $(a_0, b_0)$ such that $(2-t)/(Rt)<|a_0|<R$ and both (\ref{95}) and  (\ref{96}) become equalities. Indeed, assuming such $a_0$ and $b_0$ exist, let $T_0$ denote the corresponding operator and consider any $\epsilon>0.$ We will then have that $T_0\in \mathcal{C}_{s, t}(R)$ and also that 
$$R|b_0|-(s-\epsilon)R^2+(2-(s-\epsilon))|a_0|^2+2((s-\epsilon)-1)R|a_0|$$
$$=R|b_0|-sR^2+(2-s)|a_0|^2+2(s-1)R|a_0|+\epsilon (R^2+|a_0|^2-2R|a_0|) $$
$$=\epsilon(R-|a_0|)^2>0,$$
which shows that $ T_0\notin\mathcal{C}_{s-\epsilon, t'}(R),$ for any $t'>0$. An analogous argument involving the second inequality also shows that $T_0\notin\mathcal{C}_{s', t-\epsilon}(R)$ for any $s'>0$, as desired. \par 
Now, assume that we have equality in both (\ref{95}) and (\ref{96}) (without any extra restrictions). We can then extract a quadratic equation involving $|a|$ only: 
\begin{equation}\label{99}
(s-tR^2-2)|a|^2+2(2-s-t)R|a|+2-t+sR^2=0   
\end{equation}
Notice that, if we are able to find a solution $|a_0|$ of the above equation that also satisfies $(2-t)/(Rt)<|a_0|<R,$  our proof will be complete. Indeed, $|b_0|$ will then be uniquely determined as $$ |b_0|=sR+\frac{(s-2)|a_0|^2}{R}+ 2(1-s)|a_0|   
=t|a_0|^2R+\frac{t-2}{R}+2(t-1)|a_0|$$
and the associated $T_0$ will have the desired properties.
\par Assume first that $s=tR^2+2.$ (\ref{99}) then turns into the linear equation  $$2(2-s-t)R|a|+2-t+sR^2=0,$$ where $s=tR^2+2$. Let $a_0$ be any number satisfying the previous equation, i.e. 
     $$|a_0|=\frac{2(R^2+1)+t(R^4-1)}{2R(R^2+1)t}.$$
     One can then verify that $|a_0|<R$ is equivalent to the inequality 
     $\frac{2}{R^2+1}<t,$
     which is equivalent to $(2-t)/(Rt)<R,$ hence it must be true. An analogous argument applies for the inequality $|a_0|>(2-t)/(Rt).$
\par Assume now that $s\neq tR^2+2.$ After some calculations, we obtain (for equation (\ref{99})) the (non-negative) discriminant
$$\Delta=4[(R^2+1)^2st-2(R^2+1)(s+t)+4(R^2+1)].$$
Choose $a_0$ to be any number satisfying 
$$|a_0|=\frac{2R(s+t-2)-\sqrt{\Delta}}{2(s-tR^2-2)}.$$
Suppose first that $s>tR^2+2$. The inequality $|a_0|<R$ is equivalent to
$$tR^2+2=\frac{R^2(R^2+1)^2t^2+2(R^2+1)t-4(R^2+1)}{(R^2+1)((R^2+1)t-2)}<s,$$
hence it must be true. The proof in the case that $s<tR^2+2$ is entirely analogous. \par 
We now prove that $|a_0|>(2-t)/(tR).$ The computations here are somewhat more unpleasant. First, we assume that $s>tR^2+2$, in which case $|a_0|>(2-t)/(tR)$ is equivalent to the inequality
$$\frac{[(R^2+1)t-2]^2}{R^2}s^2+\Big(-\frac{16}{R^2}+\frac{8}{R^2}(R^2+2)t+2(R^2+1)[1-2/R^2]t^2-(R^2+1)^2t^3 \Big)s$$ $$+\frac{16}{R^2}-\frac{16}{R^2}t+4(1/R^2-R^2-1)t^2+2(R^2+1)t^3>0.$$
 But this last inequality can be rewritten as
$$\frac{1}{R^2}((R^2+1)t-2)(s-tR^2-2)([(R^2+1)t-2]s+2(2-t))\ge 0,$$ which holds for $s>tR^2+2$ and $(R^2+1)>(R^2+1)t>2,$ as desired. The proof in the case that $s<tR^2+2$ is entirely analogous. \end{proof}

 \par 
\subsection{Basic properties of $\mathbb{DLA}_R(c)$}
We now define a new operator class attached to the annulus $A_R.$ 
\begin{definition}\label{DLAdef}
Let $c\in\mathbb{R}$ and $R>1$. $\mathbb{DLA}_R(c)$ denotes the class of all operators $T\in\mathcal{B}(H)$ such that 
\begin{itemize}
    \item[(i)] $\sigma(T)\subset\overline{A_R}$ and 
    \item[(ii)] $2\Re\big[(1-zT/R)^{-1}+(1-wT^{-1}/R)^{-1}\big]-2+c\ge 0, \hspace{0.4 cm} \forall z, w\in\mathbb{D}.$
\end{itemize}
\end{definition}
 First, we prove a few elementary properties of $\mathbb{DLA}_R(c)$, including membership criteria for normal operators. 
 \begin{lemma} \label{CA, C(2-A)}
 Let $c>-2,$ $R>1$ and assume $A\in\mathcal{B}(H)$ is a self-adjoint operator such that $A>0$ and $2+c-A>0.$ If $T\in\mathcal{B}(H)$ is such that $T/R\in\mathcal{C}_{2+c-A}$ and $T^{-1}/R\in\mathcal{C}_A,$ then $T\in\mathbb{DLA}_R(c).$
 \end{lemma}
\begin{proof}
Assume $T\in\mathcal{B}(H)$ satisfies our hypotheses. Lemma \ref{CAchar} allows us to deduce that $\sigma(T/R), \  \sigma(T^{-1}/R)\subset\overline{\mathbb{D}}$, hence $\sigma(T)\subset\overline{A_R}$. Also, (\ref{CAdescr}) gives us 
$$2\Re(1-zT/R)^{-1}+2+c-A-2\ge 0 $$
and $$2\Re(1-wT^{-1}/R)^{-1}+A-2\ge 0, $$
for all $z, w\in\mathbb{D}.$ Adding these two inequalities concludes the proof.
\end{proof}
\begin{proposition}\label{DLAstuff}
Let $c\in\mathbb{R}$, $R>1$ and assume $N\in\mathcal{B}(H)$ is normal.
\begin{itemize}
    \item[(i)] If $c>-2,$ then $\mathbb{DLA}_R(c)$ will be non-empty if and only if there exists $s\in (0, 2+c)$ such that $\mathcal{C}_{2+c-s, s}(R)$ is non-empty. \\
    In particular, if $c\ge 0,$ then $\mathbb{DLA}_R(c)$  will be non-empty for every $R>1.$
    \item[(ii)] If $c\le -2,$ then $\mathbb{DLA}_R(c)$ will be empty for every $R>1.$
    \item[(iii)] $\mathbb{DLA}_R(c)\subset \mathcal{C}_{2-c, 2-c}(R),$ for every $c>-2. $
    \item[(iv)] Assume $c>-2.$ Then, $N\in \mathbb{DLA}_R(c)$ if and only if 
    $$\sigma(N)\subset \bigcup_{0<s<2+c}\mathcal{C}_{2+c-s, s}(R).$$
    \item[(v)] If, in addition, we assume $c\ge 0,$ then $N\in \mathbb{DLA}_R(c)$ is equivalent to $R^{-2}\le N^*N\le R^2.$
\end{itemize}
\end{proposition}
\begin{proof}
First, assume $c>-2.$ Lemma \ref{CA, C(2-A)} implies that $\mathcal{C}_{2+c-s, s}(R)\subset \mathbb{DLA}_R(c)$ for every $s\in (0, 2+c),$ hence one direction is obvious. For the converse, assume $\mathbb{DLA}_R(c)$ is non-empty. Thus, there exists $T\in\mathcal{B}(H)$ such that 
\begin{equation} \label{3111}
 \Big\langle \Big(2\Re\big[(1-zT/R)^{-1}+(1-wT^{-1}/R)^{-1}\big]-2+c   \Big)h, h\Big\rangle \ge 0,   
\end{equation}
for all $z, w\in\mathbb{D}$ and $x\in H.$ Now, let $\lambda\in\mathbb{C}^*$ be in the approximate point spectrum of $T$ (this is always non-empty, as it contains the topological boundary of $\sigma(T)$; see \cite[Problem 63]{Halmosproblembook}). Thus, there exists a sequence $\{h_n\}\subset H$ such that $||h_n||=1$ ($n=1, 2,\dots$) and $(\lambda-T)h_n\to 0$ as $n\to\infty$. From this last limit we easily obtain
$$
\lim_n\big(\lambda^k-T^k\big)h_n=0,$$
for all $k\in\mathbb{Z}.$  Thus, we can write $\langle T^kh_n, h_n\rangle \to \lambda^k,$ for all $k\in\mathbb{Z},$ and so 
\begin{equation}\label{3113}
    \bigg\langle \sum_{k=0}^m  (zT/R)^k h_n, h_n\bigg\rangle\to\sum_{k=0}^m(z\lambda/R) ^k
\end{equation}
and \begin{equation} \label{3114}
 \bigg\langle \sum_{k=0}^m  (wT^{-1}/R)^{k} h_n, h_n\bigg\rangle\to\sum_{k=0}^m(w/(R\lambda))^{k},   
\end{equation}
as $n\to\infty$, for all $m\ge 0$ and $z, w\in\mathbb{D}.$ \\ Now, fix $z, w\in\mathbb{D}$. We have convergence $\sum_{k=0}^m (zT/R)^k\to (1-zT/R)^{-1}$ and $\sum_{k=0}^m (wT^{-1}/R)^k\to (1-wT^{-1}/R)^{-1}$ as $m\to\infty$ in the operator norm, since $\sigma(zT/R), \sigma(wT^{-1}/R)\subset\mathbb{D}$. Hence, in view of (\ref{3113})-(\ref{3114}), we can conclude that 
$$
\lim_{n\to\infty} \big\langle (1-zT/R)^{-1}h_n, h_n\big\rangle=\lim_{m, n\to\infty} \bigg\langle \sum_{k=0}^m  (zT/R)^k h_n, h_n\bigg\rangle$$
\begin{equation} \label{3115}
=(1-z\lambda/R)^{-1}
\end{equation}
and $$
\lim_{n\to\infty} \big\langle (1-wT^{-1}/R)^{-1}h_n, h_n\big\rangle=\lim_{m, n\to\infty} \bigg\langle \sum_{k=0}^m  (wT^{-1}/R)^k h_n, h_n\bigg\rangle$$
\begin{equation} \label{3116} =(1-w/(R\lambda))^{-1}.
\end{equation}
We now set $h=h_n$ in (\ref{3111}) and let $n\to\infty.$ In view of the real-part versions of (\ref{3115})-(\ref{3116}), we obtain 
$$2\Re\big[(1-z\lambda/R)^{-1}+(1-w/(\lambda R))^{-1}\big]-2+c\ge 0,$$
for all $z, w\in\mathbb{D}.$
This last inequality can also be written as 
$$\inf_{z\in\mathbb{D}}\Big[2\Re(1-z\lambda/R)^{-1}+c\Big]\ge \sup_{w\in\mathbb{D}}\Big[-2\Re(1-w/(\lambda R))^{-1}+2\Big]$$
and thus we can deduce the existence of $s\in\mathbb{R}$ such that 
$$2\Re(1-z\lambda/R)^{-1}+(1-w/(\lambda R))^{-1}+c\ge s\ge  -2\Re(1-w/(\lambda R))^{-1}+2,$$
for all $z, w\in\mathbb{D}.$  In view of Lemma \ref{CAchar}, we obtain $\lambda/R\in\mathcal{C}_{2+c-s}$ and $1/(\lambda R)\in\mathcal{C}_{s}$, which concludes the proof.
 If $c\ge 0,$ then we can always choose  $s=1$, as $\mathcal{C}_{2+c-1, 1}(R)$ will be non-empty by Lemma \ref{51}(i). \par  For (ii), assume $c\le -2$ and let $T\in \mathbb{DLA}_R(c).$ Set $w=0$ in the definition of $\mathbb{DLA}_R(c)$ to obtain $T/R\in\mathcal{C}_{c+2}$, which implies that $c=-2$ (since $\mathcal{C}_{\rho}$ is empty for $\rho<0$). But then, we obtain $T/R\in\mathcal{C}_0$ and so $T\equiv 0$, which contradicts the inclusion $\sigma(T)\subset\overline{A_R}.$ Thus, $\mathbb{DLA}_R(c)$ must be empty. \par 
 For (iii), assume $c>-2$ and let $T\in \mathbb{DLA}_R(c).$ Setting  $w=0$ in the definition of $\mathbb{DLA}_R(c)$ gives us $T/R\in\mathcal{C}_{2-c}$, while setting $z=0$ (and letting $w\in\mathbb{D}$) gives us $T^{-1}/R\in\mathcal{C}_{2-c}$. Thus, $\mathbb{DLA}_R(c)\subset \mathcal{C}_{2-c, 2-c}(R).$ \par 
 For (iv), let $c>-2$. Since $N$ is normal, it has a spectral decomposition (see \cite[Chapter IX]{Conway})
  $$N=\int_{\sigma(N)}\lambda \ dE(\lambda).$$
  Using this, we obtain that $N\in\mathbb{DLA}_R(c)$ if and only if its spectrum is contained in $\overline{A_R}$ and 
  $$2\Re\big[(1-zN/R)^{-1}+(1-wN^{-1}/R)^{-1}\big]-2+c $$
  $$=\int_{\sigma(N)}\Big(2\Re\big[(1-z\lambda /R)^{-1}+(1-w/(\lambda R))^{-1}\big]-2+c\Big)\ dE(\lambda)\ge 0,$$
  for all $z, w\in\mathbb{D},$ which can be equivalently restated as 
  $$2\Re\big[(1-z \lambda /R)^{-1}+(1-w/(\lambda R))^{-1}\big]-2+c\ge 0,$$
  for all $z, w\in\mathbb{D}$ and all $\lambda\in\sigma(N).$  Mimicking our argument from the proof of (i), we deduce, for every $\lambda\in\sigma(N)$, the existence of $s\in (0, 2+c)$ (depending on $\lambda$) such that $\lambda\in\mathcal{C}_{2+c-s, s}(R)$. Hence, we can write 
  $$\sigma(N)\subset \bigcup_{0<s<2+c}\mathcal{C}_{2+c-s, s}(R),$$
  which concludes the proof. \par 
  For (v), assume $c\ge 0$.  Since $N$ is normal, the condition $R^{-2}\le N^*N\le R^2$ is equivalent to $\sigma(N)\subset A_R$, which is necessary for membership in $\mathbb{DLA}_R(c).$ Conversely, if $R^{-2}\le N^*N\le R^2$, then 
  $$\sigma(N)\subset \mathcal{C}_{1, 1}(R)\subset \bigcup_{0<s<2+c}\mathcal{C}_{2+c-s, s}(R),$$
  where the last inclusion holds because $c\ge 0$. Thus, we must have $N\in\mathbb{DLA}_R(c)$. 
\end{proof}

Under the extra assumption $\sigma(T)\subset A_R,$ we can restrict the parameters $z, w$ in the definition of $\mathbb{DLA}_R(c)$ to the boundary of the disk (this is the $\mathbb{DLA}_R(c)$-version of Lemma \ref{fromztotheta}).
\begin{lemma} \label{DLAfromztotheta}
    Assume $c>-2$ and $T\in\mathcal{B}(H)$ is such that $\sigma(T)\subset A_R.$ Then, $T\in \mathbb{DLA}_R(c)$ if and only if 
 $$2\Re\big[(1-e^{i\theta}T/R)^{-1}+(1-e^{i\psi}T^{-1}/R)^{-1}\big]-2+c\ge 0, \hspace{0.4 cm} \forall \theta, \psi\in [0, 2\pi).$$
\end{lemma}
\begin{proof}
First, assume $T\in \mathbb{DLA}_R(c)$. Fix an arbitrary $h\in H$ and $w\in\mathbb{D}$ and define 
$$\Phi_{h, w}: \mathbb{D}\to \mathbb{C} $$
$$z\mapsto \big\langle \big( 2\Re\big[(1-zT/R)^{-1}+(1-wT^{-1}/R)^{-1}\big]-2+c\big)h, h\big\rangle.$$
Since $\sigma(T)\subset A_R$, $\Phi_{h, w}$ will be a harmonic function on $\mathbb{D}$ that extends continuously to $\overline{\mathbb{D}}$. By the minimum principle for harmonic functions, we then obtain that 
 $$\Phi_{h, w}(z)=\big\langle \big( 2\Re\big[(1-zT/R)^{-1}+(1-wT^{-1}/R)^{-1}\big]-2+c\big)h, h\big\rangle\ge 0,  $$ for all $z\in\mathbb{D}$, 
 if and only if 
 $$\Phi_{h, w} (e^{i\theta})=\big\langle \big( 2\Re\big[(1-e^{i\theta}T/R)^{-1}+(1-wT^{-1}/R)^{-1}\big]-2+c\big)h, h\big\rangle\ge 0, $$ for all $\theta\in [0, 2\pi).$
 Since $h$ and $w$ were arbitrary, we conclude that
$$2\Re\big[(1-e^{i\theta}T/R)^{-1}+(1-wT^{-1}/R)^{-1}\big]-2+c\ge 0, $$
for all $\theta\in [0, 2\pi)$ and all $w\in\mathbb{D}.$ We can now apply the minimum principle to the function 
$$w\mapsto \big\langle \big( 2\Re\big[(1-e^{i\theta}T/R)^{-1}+(1-wT^{-1}/R)^{-1}\big]-2+c\big)h, h\big\rangle$$
to conclude the proof. \par 
For the converse, simply roll back the steps in the previous proof. 
\end{proof}
Now, recall that the $\mathcal{C}_{\rho}$ classes are strictly monotone with respect to $\rho,$ i.e. $\rho<\rho'$ implies $\mathcal{C}_{\rho}\subsetneq \mathcal{C}_{\rho'}$. We are going to prove an analogous monotonicity result for $\mathbb{DLA}_R(c).$
\begin{theorem}\label{monot}
If $-2<c<c'$ and $\mathbb{DLA}_R(c')$ is non-empty, then $$\mathbb{DLA}_R(c)\subsetneq \mathbb{DLA}_R(c').$$
\end{theorem}
\begin{proof} If $c<c'$, it is obvious by the definition of $\mathbb{DLA}_R(c)$ that $\mathbb{DLA}_R(c)\subseteq\mathbb{DLA}_R(c').$ To show that the inclusion is actually strict, we are going to divide the proof into two cases. \par 
First, assume that there exist $s, t>0$ such that $s\ge 1, s+t=2+c'$ and $\mathcal{C}_{s. t}(R)$ is non-empty. Note that (in view of Lemma \ref{51} and the monotonicity of the $\mathcal{C}_{\rho}$ classes), if such $s, t$ exist and $t\ge 1,$ we can replace them by new parameters $s', t'$ such that $s'\ge 1$, $t'<1$, $s'+t'=2+c'$, $ \frac{2}{R^2+1}<t$ and $\mathcal{C}_{s'. t'}(R)$ is non-empty. So, we may also assume that $t<1$ and $\frac{2}{R^2+1}<t$. \par 
    Now, we will take advantage of the matrices we calculated in the proof of Theorem \ref{110}. Recall first that, in view of Lemma \ref{90}, a matrix $$T=\begin{pmatrix} a & b \\ 0 & a \end{pmatrix}
     $$ with $a, b>0$ lies in $\mathcal{C}_{s, t}(R)$ if and only if $(2-t)/(Rt)\le a\le R$ and 
     \begin{equation}\label{300}
      2\Re \bigg[\frac{1}{1-ae^{i\theta}/R} \bigg] +(s-2)\ge \frac{b}{R|1-ae^{i\theta}/R|^2}  
     \end{equation}
 holds for $\theta=0$ and 
  \begin{equation}\label{302}
      2\Re \bigg[\frac{1}{1-e^{i\psi}/(aR)} \bigg] +(t-2)\ge \frac{b}{a^2}\frac{1}{R|1-e^{i\psi}/(aR)|^2}  
     \end{equation}
     holds for $\psi=\pi$. Since $\frac{2}{R^2+1}<t$, the proof of Theorem \ref{110} tells us that we can find $a, b$ such that $(2-t)/(Rt)<a<R$ and we have equality in (\ref{300}) for $\theta=0$ and in (\ref{302}) for $\psi=\pi.$ Since, for these values of $a, b$, we have $T\in\mathcal{C}_{s, t}(R)$ and $s+t=2+c',$ Lemma \ref{CA, C(2-A)} tells us that $T\in\mathbb{DLA}_R(c')$. \par Now, we claim that $T$ cannot lie in $\mathbb{DLA}_R(c)$. Indeed, assume $T\in \mathbb{DLA}_R(c)$. Taking determinants in the inequality
   $$2\Re\big[(1-zT/R)^{-1}+(1-wT^{-1}/R)^{-1}\big]-2+c\ge 0 $$
     implies that 
$$2\Re\bigg[\frac{1}{1-ae^{i\theta}/R}+\frac{1}{1-e^{i\psi}/(aR)} \bigg]-2+c$$ $$\ge \bigg|\frac{be^{i\theta}}{R(1-ae^{i\theta}/R)^2}-\frac{b}{a^2}\frac{e^{i\psi}}{R(1-e^{i\psi}/(aR))^2} \bigg| $$
     for all $\theta$ and $\psi.$ In particular, we can choose $\theta_0=0$ and $\psi_0=\pi,$ which implies (since we have equality in both (\ref{300}) and (\ref{302}))
$$2\Re\bigg[\frac{1}{1-ae^{i\theta_0}/R}+\frac{1}{1-e^{i\psi_0}/(aR)} \bigg]-2+c$$ $$\ge \bigg|\frac{be^{i\theta_0}}{R(1-ae^{i\theta_0}/R)^2}-\frac{b}{a^2}\frac{e^{i\psi_0}}{R(1-e^{i\psi_0}/(aR))^2} \bigg|$$
$$=\bigg|\frac{be^{i\theta_0}}{R(1-ae^{i\theta_0}/R)^2}\bigg|+\bigg|\frac{b}{a^2}\frac{e^{i\psi_0}}{R(1-e^{i\psi_0}/(aR))^2} \bigg| $$
$$=2\Re \bigg[\frac{1}{1-ae^{i\theta_0}/R} \bigg] +(s-2)+ 2\Re \bigg[\frac{1}{1-e^{i\psi_0}/(aR)} \bigg] +(t-2)$$
$$=2\Re \bigg[\frac{1}{1-ae^{i\theta_0}/R} \bigg]+ 2\Re \bigg[\frac{1}{1-e^{i\psi_0}/(aR)} \bigg]+c'-2, $$
a contradiction, since $c<c'$. Thus,  $T\notin\mathbb{DLA}_R(c')$. \par  
  Now, assume that there do not exist $s, t>0$ such that $s\ge 1, s+t=2+c'$ and $\mathcal{C}_{s. t}(R)$ is non-empty.  Since $T\in\mathcal{C}_{s, t}(R)$ if and only if $T^{-1}\in\mathcal{C}_{t, s}(R)$, we also deduce that there do not exist $s, t>0$ such that $t\ge 1, s+t=2+c'$ and $\mathcal{C}_{s. t}(R)$ is non-empty. Thus, $-2<c'<0$. Now, $\mathbb{DLA}_R(c')$ is non-empty, so Proposition \ref{DLAstuff}(i) tells us that there exists $s\in (0, 2+c')$ such that $\mathcal{C}_{s, 2+c'-s}(R)$ is non-empty. In view of our previous remarks, we must have $s, 2+c'-s<1$. Lemma \ref{51} then tells us that 
  $$\frac{1}{R}\frac{s-c'}{2+c'-s}\le R\frac{s}{2-s}.$$
  Note that the right-hand side tends to $0$ as $s\to 0$, while the left-hand side remains bounded below  by a strictly positive number. Thus, shrinking $s$, we may replace it by a positive number $\delta$ such that $\frac{1}{R}\frac{\delta-c'}{2+c'-\delta}= R\frac{\delta}{2-\delta}.$ In view again of Lemma \ref{51}, we obtain that $\mathcal{C}_{\delta, 2+c'-\delta}(R)$ is non-empty. Setting $a=R\frac{\delta}{2-\delta},$ we obtain, from the same Lemma, that $a\in \mathcal{C}_{\delta, 2+c'-\delta}(R)\subset\mathbb{DLA}_R(c')$.  
\par Now, we shall show that $a\notin \mathbb{DLA}_R(c)$. Assume instead that $a\in \mathbb{DLA}_R(c)$.  Proposition \ref{DLAstuff} then implies that we can find $s_0\in (0, 2+c)$ such that $a\in\mathcal{C}_{s_0, 2+c-s_0}(R)$. But $c<c',$ hence we must have either $s_0<\delta$ or $2+c-s_0<2+c'-\delta$. Assume that $s_0<\delta,$ then 
$$R\frac{s_0}{2-s_0}<R\frac{\delta}{2-\delta}=a,$$
which contradicts $a\in\mathcal{C}_{s_0, 2+c-s_0}(R)$. Similarly, if $2+c-s_0<2+c'-\delta$ we can write 
$$\frac{1}{R}\frac{2-(2+c-s_0)}{2+c-s_0}>\frac{1}{R}\frac{2-(2+c'-\delta)}{2+c'-\delta}=\frac{1}{R}\frac{\delta-c'}{2+c'-\delta}=a,$$
which again contradicts $a\in\mathcal{C}_{s_0, 2+c-s_0}(R)$. Thus, $a\notin \mathbb{DLA}_R(c)$ and we are done.
\end{proof}

It is well-known (see e.g. \cite{Williams}) that given any $T\in\mathcal{C}_{\rho}$ ($\rho>0$) and any $f\in\mathcal{A}(\DD)$ such that $||f||_{\infty}\le 1$ and $f(0)=0,$ we must have $f(T)\in\mathcal{C}_{\rho}.$ We end this subsection with a proposition that is motivated by this result.

\begin{proposition} 
Assume $c>-2$ and $T\in \mathbb{DLA}_R(c)$. Then, for any  $f, g\in \mathcal{A}(\DD)$ that are bounded by $1$ and satisfy $f(0)=g(0)=0,$ we have
 $$2\Re\big[(1-zf(T/R))^{-1}+(1-wg(T^{-1}/R))^{-1}\big]-2+c\ge 0, $$
 for all $z, w\in\mathbb{D}.$
\end{proposition}
\begin{proof}
By assumption, we know that 
 \begin{equation} \label{80} \Re\bigg[\sum_{n=1}^{\infty}\bigg(\frac{Tz}{R}\bigg)^n+1+\sum_{n=1}^{\infty}\bigg(\frac{T^{-1}w}{R}\bigg)^{n}\bigg]\ge \frac{-c}{2} ,\end{equation}
 for all $z, w\in\mathbb{D}.$ Fix $w\in\mathbb{D}$, let $x\in H$ and choose a decreasing null sequence $\{\epsilon_k\}.$ Set 
 $$S_k=S_k(w)=1+c/2+\epsilon_k+\Re\sum_{n=1}^{\infty}\bigg(\frac{T^{-1}w}{R}\bigg)^{n}\in\mathcal{B}(H).$$
 (\ref{80}) now tells us that the holomorphic function 
 $$F_k:\mathbb{D}\to\mathbb{C} $$
 $$z\mapsto \sum_{n=1}^{\infty}\frac{z^n}{R^n}\langle T^nx,x\rangle +\langle S_k x, x \rangle $$
 has positive real part and satisfies $F_k(0)=\langle S_k x, x \rangle>0$ (put $z=0$ in (\ref{80})). Thus, Herglotz's theorem implies the existence of a positive measure $\mu_{x, w, k}$ on the unit circle such that 
 $$F_k(0)+\sum_{n=1}^{\infty}\frac{z^n}{R^n}\langle T^nx,x\rangle=\int \frac{1+ze^{-i\theta}}{1-ze^{-i\theta}}d\mu_{x, w, k}(\theta),  $$ for all $z\in\mathbb{D}.$ Expanding the integrand and equating coefficients, we obtain 
 $$\frac{1}{R^n} \langle T^nx,x\rangle=2\int e^{-in\theta}d\mu_{x, w, k}(\theta),$$
 for every $n\ge 1.$ Thus, if $p$ is any polynomial such that $p(0)=0$, we can write 
 $$\langle p(T/R)x, x\rangle=2\int p(e^{-i\theta})d\mu_{x, w, k}(\theta).$$
Replacing $p$ by $p^n,$ we obtain
$$\langle p^n(T/R)x, x\rangle=2\int p^n(e^{-i\theta})d\mu_{x, w, k}(\theta),$$
for every $n\ge 1.$
Thus, if we also assume that $|p|\le 1,$ we obtain that $p(T/R)$ has its spectrum inside $\mathbb{D}$ (since the same must be true for $T/R$) and we can write 
$$\bigg\langle\bigg(S_k+ \sum_{n=1}^{\infty}z^n p^n(T/R) \bigg)x,x\bigg\rangle$$ $$=F_k(0)+2\sum_{n=1}^{\infty} z^n\int p(e^{-i\theta})^nd\mu_{x, w, k}(\theta)=\int \frac{1+zp(e^{-i\theta})}{1-zp(e^{-i\theta})}d\mu_{x, w, k}(\theta).$$
Now, if  $f\in\mathcal{A}(\DD)$ is bounded by $1$ and satisfies $f(0)=0,$ a standard approximation argument shows that 
$$\bigg\langle\bigg(S_k+ \sum_{n=1}^{\infty}z^n f^n(T/R) \bigg)x,x\bigg\rangle=\int \frac{1+zf(e^{i\theta})}{1-zf(e^{i\theta})}d\mu_{x, w, k}(\theta).$$
The integrand has positive real part for all $z$ and $\theta$, hence
$$\Re\bigg[\sum_{n=1}^{\infty}\big(zf(T/R)\big)^n+1+\epsilon_k+\sum_{n=1}^{\infty}\big(wT^{-1}/R\big)^{n}\bigg]\ge \frac{-c}{2},$$
for all $k$ and for all $z, w\in\mathbb{D}.$ Letting $k\to\infty$, we obtain 
$$\Re\bigg[\sum_{n=1}^{\infty}\big(zf(T/R)\big)^n+1+\sum_{n=1}^{\infty}\big(wT^{-1}/R\big)^{n}\bigg]\ge \frac{-c}{2},$$
for all $z, w\in\DD.$
Using this last inequality, we may repeat the previous argument with the roles of $z$ and $w$ swapped, thus obtaining the desired result.
\end{proof}

 \par 
\subsection{The double-layer potential kernel} In this subsection, we exhibit the connection between $\mathbb{DLA}_R(c)$ and the double-layer potential kernel over the annulus, as described in the introduction. \par 
Recall that, given any bounded open $\Omega\subset\mathbb{C}$ and any $T\in\mathcal{B}(H)$ such that $\sigma(T)\subset \Omega,$ we may consider the Cauchy transforms of $f$ and $\overline{f}$
$$f(T)=(Cf)(T)=\frac{1}{2\pi i}\int_{\partial \Omega}f(\sigma)(\sigma -T)^{-1}d\sigma, $$ $$(C\overline{f})(T)=\frac{1}{2\pi i}\int_{\partial \Omega}\overline{f(\sigma)}(\sigma -T)^{-1}d\sigma.$$
We also define the transform of $f$ by the double-layer potential kernel
$$S(f, T)=\int_{\partial\Omega}\mu(\sigma(s), T)f(\sigma(s))ds,$$
where $s$ denotes the arc length of $\sigma=\sigma(s)$ on the (counter-clockwise) oriented boundary $\partial\Omega$ and $\mu(\sigma(s), T)$ is the self-adjoint operator defined (for $\sigma(s)\notin\sigma(T))$ as 
$$\mu(\sigma(s), T)=\frac{1}{2\pi i}\big(\sigma'(s)(\sigma(s)-T)^{-1}-\overline{\sigma'(s)}(\overline{\sigma(s)}-T^*)^{-1}\big).$$
Note that $S(f, T)=f(T)+(C\overline{f})(T)^*$ and thus $\int_{\partial\Omega}\mu(\sigma, T)ds=2I.$ \par 
The definition of $\mathbb{DLA}_R(c)$ can now be recast (for $\sigma(T)\subset A_R$) as follows. 
\begin{theorem} \label{doublelayer}
 Assume $ R>1, c>-2$ and $T\in\mathcal{B}(H)$ satisfies $\sigma(T)\subset A_R.$ Then, $T\in \mathbb{DLA}_R(c)$ if and only if the mapping 
 $$S_{R, c}: \mathcal{A}(A_R)\to\mathcal{B}(H) $$
 $$f=\sum_{n\in\mathbb{Z}}a_n z^n \mapsto \frac{1}{2+c}\bigg[ \int_{\partial A_R}\mu(\sigma, T)f(\sigma)ds +ca_0\bigg] $$
 is contractive.
\end{theorem}
\begin{proof}
 Write $\partial A_R=C_1\cup C_{-1},$ where $C_1$ (the outer circle) is counter-clockwise oriented, while $C_{-1}$ (the inner circle) is clockwise oriented. Also, in view of \cite[Corollary 2.9]{Paulsenbook}
 and \cite[Proposition 2.12]{Paulsenbook}, the fact that $S_{R, c}$ is unital allows us to deduce that $S_{R, c}$ is contractive if and only if 
 \begin{equation} \label{sprime}
     S'_{R, c}: \mathcal{A}(A_R)+\mathcal{A}(A_R)^*\to\mathcal{B}(H)\end{equation}
 $$f+\overline{g}=\sum_{n\in\mathbb{Z}}a_n z^n+\sum_{n\in\mathbb{Z}}\overline{b_n} \overline{z}^n\mapsto S_{R, c}(f)+S_{R, c}(g)^*$$ $$= \frac{1}{2+c}\bigg[ \int_{\partial A_R}\mu(\sigma, T)(f+\overline{g})(\sigma)ds +c(a_0+\overline{b_0})\bigg] $$
 is positive. Let $C(\partial A_R)$ denote the algebra of continuous functions on $\partial A_R$ and recall that the closure of $\mathcal{A}(A_R)+\mathcal{A}(A_R)^*$ in $C(\partial A_R)$ is the codimension one subspace (see \cite[p. 80]{Paulsenbook}) 
 $$\mathcal{M}_R=\Big\{f\in C(\partial A_R): \frac{1}{2\pi R}\int_{C_1} f(\sigma)ds=\frac{R}{2\pi}\int_{C_{-1}} f(\sigma)ds  \Big\}.$$
Thus, our goal will be to show that $T\in\mathbb{DLA}_R(c)$  if and only if $\tilde{S}_{R, c}$ is positive over $\mathcal{M}_R$, where 
$$\tilde{S}_{R, c}: \mathcal{M}_R\to\mathcal{B}(H) $$
 $$f\mapsto \frac{1}{2+c}\bigg[ \int_{\partial A_R}\mu(\sigma, T)f(\sigma)ds +\frac{c}{2\pi i}\int_{C_1}\frac{f}{\sigma}\sigma'ds\bigg]$$
 (note that $S'_{R, c}\equiv \tilde{S}_{R, c}$ over $\mathcal{A}(A_R)+\mathcal{A}(A_R)^*$). We require the  following lemma.
 \begin{lemma}\label{somelemma}
In the setting of Theorem \ref{doublelayer}, $T\in\mathbb{DLA}_R(c)$ if and only if 
$$ R\mu(\sigma_1, T) + R^{-1}\mu(\sigma_2, T)+\frac{c}{2\pi}\ge 0,$$
for all $\sigma_1\in C_1$ and $\sigma_2\in C_{-1}.$
 \end{lemma}
 \begin{proof}[Proof of Lemma \ref{somelemma}]
 For $\sigma_1\in C_1,$ the arc-length parametrization gives $\sigma_1=Re^{i\theta}, \sigma'_1=ie^{i\theta}$, while for $\sigma_2\in C_{-1}$ we obtain $\sigma_2=R^{-1}e^{-i\psi}, \sigma'_2=-ie^{-i\psi},$ where $\theta, \psi\in [0, 2\pi).$ Thus, we may write 
 $$2\pi\big[R\mu(\sigma_1, T) + R^{-1}\mu(\sigma_2, T)\big]+c$$
 $$=2\Re \big[ Re^{i\theta}(Re^{i\theta}-T)^{-1}\big]+2\Re \big[ -R^{-1}e^{-i\psi}(R^{-1}e^{-i\psi}-T)^{-1}\big]+c  $$
 $$=2\Re \big[(1-e^{-i\theta}T/R)^{-1}+(1-e^{-i\psi}T^{-1}/R)^{-1}              \big]+c-2,$$
 which is positive for all $\theta, \psi\in [0, 2\pi)$ if and only if $T\in\mathbb{DLA}_R(c)$.
 \end{proof}
Now, assume that $\tilde{S}_{R, c}$ is positive and that $T\notin\mathbb{DLA}_R(c)$. In view of Lemma \ref{somelemma}, there exist $\eta_1\in C_1$,  $\eta_2\in C_{-1}$ and a unit vector $v\in H$ such that 
\begin{equation}\label{negative}
\Big \langle \Big(R\mu(\eta_1, T) + R^{-1}\mu(\eta_2, T)+\frac{c}{2\pi}\Big)v, v\Big\rangle=2k<0. 
\end{equation}
Since $\sigma(T)\subset A_R,$ we know that both of the maps 
$\sigma_1\mapsto  \langle \mu(\sigma_1, T)v, v\rangle $
and $\sigma_2\mapsto \langle \mu(\sigma_2, T)v, v\rangle$ are continuous. Thus, in view of (\ref{negative}), we can find small arcs $I_1\subset C_1$ and $I_{-1}\subset C_{-1}$ of equal length and centered at $\eta_1$ and $\eta_2$ respectively such that
$$ 
\Big \langle \Big(R\mu(\sigma_1, T) + R^{-1}\mu(\sigma_2, T)+\frac{c}{2\pi}\Big)v, v\Big\rangle\le k<0, 
$$ 
for all $\sigma_1\in I_1$ and all $\sigma_2\in I_{-1}.$ 
From this, we easily deduce the existence of $t\in\mathbb{R}$ such that 

\begin{equation}\label{negativeext1}
\Big \langle \Big(R\mu(\sigma_1, T) +\frac{c}{2\pi}\Big)v, v\Big\rangle\le k/2+t 
\end{equation}
and 
\begin{equation}\label{negativeext2}
\big \langle R^{-1}\mu(\sigma_2, T)v, v\big\rangle\le k/2-t, 
\end{equation}
for all $\sigma_1\in I_1$ and all $\sigma_2\in I_{-1}.$
Now, take $g: \partial A_R\to\mathbb{C}$ to be a continuous function such that $0\le g\le 1$, $g(z)=0$ for $z$ outside $I_1\cup I_{-1}$ and also $d=\frac{1}{2\pi R}\int_{C_1} g(\sigma)ds=\frac{R}{2\pi}\int_{C_{-1}} g(\sigma)ds>0$. Then,  $g\in\mathcal{M}_R$  and so $\tilde{S}_{R, c}(g)$ must be a positive operator. However, observe that by (\ref{negativeext1}) and (\ref{negativeext2}),
$$(2+c)\langle \tilde{S}_{R, c}(g) v, v\rangle=\int_{\partial A_R}\big\langle\mu(\sigma, T)v, v \big\rangle g(\sigma)ds +\frac{c}{2\pi R}\int_{C_1} g(\sigma_1)ds  $$
$$=\int_{C_1}\langle R\mu(\sigma_1, T)v, v \rangle\frac{g(\sigma_1)}{R}ds + \int_{C_{-1}}\langle R^{-1}\mu(\sigma_2, T)v, v \rangle Rg(\sigma_2)ds+\frac{c}{2\pi R}\int_{C_1} g(\sigma_1)ds $$
 $$\le (k/2+t)\frac{1}{ R}\int_{C_1} g(\sigma_1)ds+(k/2-t)R\int_{C_{-1}} g(\sigma_2)ds $$
 $$=2\pi dk<0,$$
 a contradiction. Thus, we must have $T\in\mathbb{DLA}_R(c)$. \par 
 Conversely, assume $T\in\mathbb{DLA}_R(c)$. Fix $v\in H$. Lemma \ref{somelemma} tells us that we can find $t\in\mathbb{R}$ such that 

\begin{equation}\label{negativeext3}
\Big \langle \Big(R\mu(\sigma_1, T) +\frac{c}{2\pi}\Big)v, v\Big\rangle\ge t
\end{equation}
and 
\begin{equation}\label{negativeext4}
\big \langle R^{-1}\mu(\sigma_2, T)v, v\big\rangle\ge -t, 
\end{equation}
for all $\sigma_1\in C_1$ and all $\sigma_2\in C_{-1}.$ Now, let $f\in \mathcal{M}_R$ be positive. Since $\frac{1}{2\pi R}\int_{C_1} f(\sigma)ds=\frac{R}{2\pi}\int_{C_{-1}} f(\sigma)ds$, we may write (in view of (\ref{negativeext3})-(\ref{negativeext4}))
$$(2+c)\langle \tilde{S}_{R, c}(f) v, v\rangle$$ $$=\int_{C_1}\Big \langle \Big(R\mu(\sigma_1, T) +\frac{c}{2\pi}\Big)v, v\Big\rangle \frac{f(\sigma_1)}{R}ds+\int_{C_{-1}}\big \langle R^{-1}\mu(\sigma_2, T)v, v\big\rangle Rf(\sigma_2)ds $$
$$\ge \frac{t}{R}\int_{C_1} f(\sigma_1)ds-tR\int_{C_{-1}} f(\sigma_2)ds $$
$$=0.$$
Since $v\in H$ was arbitrary, $\tilde{S}_{R, c}(f)$ has to be a positive operator and we are done.
\end{proof}

  \par \small 
{\section{THE $\mathbb{CDLA}_{R}(c)$ CLASS}\label{completeversion}}
 \large 
 In this section, we introduce and characterize $\mathbb{CDLA}_{R}(c)$, the ``complete version" of $\mathbb{DLA}_{R}(c)$.  Our main result is Theorem \ref{2002}. One can also work with operators satisfying $\sigma(T)\subset \overline{A_R}$ to obtain the more general Theorem \ref{droptheorem}.\\
 \begin{definition}
    Assume $ R>1, c>-2$ and $T\in\mathcal{B}(H)$ satisfies $\sigma(T)\subset A_R.$ Then, $T\in \mathbb{CDLA}_R(c)$ if the mapping
 $$S_{R, c}: \mathcal{A}(A_R)\to\mathcal{B}(H) $$
 $$f=\sum_{n\in\mathbb{Z}}a_n z^n \mapsto \frac{1}{2+c}\bigg[ \int_{\partial A_R}\mu(\sigma, T)f(\sigma)ds +ca_0\bigg] $$
 is completely contractive.
 \end{definition}
 
We first establish one direction of Theorem \ref{2002}.
\begin{lemma}\label{onedir}
 Assume $ R>1, c>-2$ and $T\in\mathcal{B}(H)$ satisfies $\sigma(T)\subset A_R.$ If there exists $A\in\mathcal{B}(H)$ such that $A>0, \ 2+c-A>0$ and $T/R\in \mathcal{C}_{2+c-A}$ and $T^{-1}/R\in \mathcal{C}_A $, then $T\in \mathbb{CDLA}_R(c)$.
\end{lemma}
\begin{proof}
Assume $T$ and $A$ satisfy the given hypotheses. We again write $\partial A_R=C_1\cup C_{-1},$ where $C_1$ (the outer circle) is counter-clockwise oriented, while $C_{-1}$ (the inner circle) is clockwise oriented. \par Now, define the self-adjoint operator 
    $$\nu_A(\sigma, T):=\mu(\sigma, T)+\frac{\sigma'}{\sigma}\frac{(c-A)}{2\pi i}, \hspace{0.2 cm}\forall \sigma\in\partial C_1,$$
    and  $$\nu_A(\sigma, T):=\mu(\sigma, T)-\frac{\sigma'}{\sigma}\frac{A}{2\pi i}, \hspace{0.2 cm}\forall \sigma\in\partial C_{-1}.$$
   Note that, if $\sigma\in C_1,$ we have $\sigma=Re^{i\theta}$ and $s=R\theta,$ thus $$\nu_A(\sigma, T)=\mu(Re^{i\theta}, T)+\frac{c-A}{2\pi R}\ge 0,$$ for all $\theta$, as $T/R\in \mathcal{C}_{2+c-A}$. Also, if $\sigma\in C_{-1},$ we can write $\sigma=R^{-1}e^{-i\phi}$ and $s=R^{-1}\phi,$ hence  $$\nu_A(\sigma, T)=\mu(R^{-1}e^{-i\phi}, T)+\frac{R}{2\pi}A $$ 
   $$=\frac{1}{2\pi }\big(-e^{-i\phi}(R^{-1}e^{-i\phi}-T)^{-1}+e^{i\phi}(R^{-1}e^{i\phi}-T^*)^{-1}\big)+\frac{R}{2\pi}A $$
   $$=\frac{R}{2\pi}\Big(2\Re\big(R^{-1}T^{-1}e^{-i\phi}(I-T^{-1}e^{-i\phi}R^{-1})^{-1}\big)+A\Big)$$
   $$=\frac{R}{2\pi}\big(2\Re(I-T^{-1}e^{-i\phi}R^{-1})^{-1}+A-2\big)\ge 0,$$
  for all $\phi,$ as $T^{-1}/R\in\mathcal{C}_{A}$. \par Next, we consider the coordinate-wise map $S_{R, c}^{(m)}:M_m(\mathcal{A}(A_R))\to \mathcal{B}(H^{(m)}),$ for $m\ge 1$. Here, $M_m(\mathcal{A}(A_R))$ denotes the algebra of all matrix-valued $F: A_R\to \mathbb{C}^{m\times m}$ that are (coordinate-wise) analytic and admit a continuous extension to $\overline{A_R}.$ For any such $F=\sum_{n\in\mathbb{Z}} A_n\otimes z^n,$ we can write 
   $$(2+c)S^{(m)}_{R, c}(F)=\int_{\partial A_R}F(\sigma)\otimes \mu(\sigma, T)ds+cA_0\otimes I
    $$ 
    $$=\int_{\partial A_R}F(\sigma)\otimes \mu(\sigma, T)\ ds+\frac{c}{2\pi i}\int_{C_1}\frac{\sigma'}{\sigma}F(\sigma)\otimes I\ ds $$
    $$=\int_{\partial A_R}F(\sigma)\otimes \mu(\sigma, T)\ ds+\frac{c}{2\pi i}\int_{C_1}\frac{\sigma'}{\sigma}F(\sigma)\otimes I\ ds-\frac{A}{2\pi i}\int_{\partial A_R} \frac{\sigma'}{\sigma}F(\sigma)\otimes I\ ds,$$
    $$=\int_{\partial A_R}F(\sigma)\otimes \nu_A(\sigma, T)\ ds.$$

   But now, since $\nu_{A}(\sigma, T)\ge 0$ for every $\sigma$ in $\partial A_R$ and also $\int_{A_R}\nu_{A}(\sigma, T)\ ds=(2+c)I,$ one can show (see e.g. the proof of Lemma 2.2 in \cite{bivariate}) that $S_{R, c}^{(m)}$ is contractive, for every $m\ge 1.$ This concludes the proof.
\end{proof}

 We now prove a lemma; the $\mathcal{C}_A$ classes do not contain any invertible operators if $A$ is not invertible. 
\begin{lemma}\label{notinvertCA}
Let $A\in\mathcal{B}(H)$ be a positive operator that is not invertible. Assume also that $T\in\mathcal{B}(H)$ satisfies $\sigma(T)\subset\mathbb{D}$ and 
$$  
2\Re(1-zT)^{-1}+A-2\ge 0, \hspace{0.3 cm}  \forall z\in \mathbb{D}.
$$ 
Then, $T$ is not invertible.
\end{lemma}
\begin{proof}
Arguing as in the proof of Theorem \ref{doublelayer}, one can show that $T$ satisfies 
$$  
2\Re(1-zT)^{-1}+A-2\ge 0, \hspace{0.3 cm}  \forall z\in \mathbb{D}.
$$ 
if and only if the mapping 
$$S_{A}:\mathcal{A}(\mathbb{D})\to\mathcal{B}(H)$$
$$f\mapsto \int_{\partial\mathbb{D}}\mu(\sigma, T)f(\sigma)ds +f(0)(A-2)$$
    is contractive. But then, we know (see e.g. \cite[Chapters 2-3]{Paulsenbook}) that $S_{A}$ is contractive if and only if it has a completely positive extension to all of $C(\partial\mathbb{D})$. Let $\tilde{S}_A$ denote such an extension. The non-unital version of Stinespring's Theorem \cite[Theorem 4.1]{Paulsenbook} then implies the existence of a Hilbert space $K\supset H$, a unital $*$-homomorphism $\pi: C(\partial\mathbb{D})\to\mathcal{B}(K)$ and a bounded operator $V: H\to K$ such that 
   \begin{equation}\label{rep}
       \tilde{S}_A(f)=V^*\pi(f)V, \hspace{0.3 cm} \forall f\in C(\partial\mathbb{D}).
   \end{equation}
    Set $U=\pi(z).$ It is easy to see that $U$ will then be a unitary operator. Setting $f\equiv 1$ in (\ref{rep}) gives us $A=V^*V,$ thus the polar decomposition of $V$ will be given by $YA^{1/2},$ where $Y$
is some partial isometry. But observe also that putting $f=z$ in (\ref{rep}) gives us 
$$T=A^{1/2}Y^*UYA^{1/2}.$$
Since $A^{1/2}$ is not invertible, we conclude that $T$ cannot be invertible.
\end{proof}

Before we finish off the proof of Theorem \ref{2002}, a few dilation-theoretic observations are in order. Assume $T\in \mathcal{B}(H)$, with $\sigma(T)\subset A_R$. It is then well-known that $T\in \mathbb{CDLA}_R(c)$ if and only if the mapping (\ref{sprime}) is completely positive. Further, by Stinespring's Theorem, this is equivalent to the existence of a Hilbert space $K\supset H$ and a unital $*$-homomorphism $\pi: C(\partial A_R)\to\mathcal{B}(K)$ such that, for $f=\sum a_n z^n\in \mathcal{A}(A_R)$,
$$(2+c)S_{R, c}(f)=f(T)+(C\overline{f})(T)^*+ca_0=(2+c)P_H\pi (f)|_H.$$
Set $\pi(z)=N.$ Then, $N$ will be a normal operator satisfying $\sigma(N)\subset \partial A_R$ and our previous equality becomes, for $f(z)=z^n$, 
$$T^n+(C\overline{z^n})(T)^*=(2+c)P_H N^n|_H, \hspace*{0.3 cm} \forall n\neq 0.$$
After some computations, one verifies that $$(C\overline{z^n})(\zeta)=\frac{1}{2\pi i}\int_{\partial A_R}\frac{\overline{z^n}}{z-\zeta}\ dz=R^{-2|n|}\zeta^{-n},$$
hence
\begin{equation} \label{naivedil}
  T^n+R^{-2|n|}T^{-n*}=(2+c)P_H N^n|_H, \hspace*{0.3 cm} \forall n\neq 0.
\end{equation}
On the other hand, Theorem \ref{2002} tells us that $T\in \mathbb{CDLA}_R(c)$ if and only if there exists $A\in\mathcal{B}(H)$ such that $A>0, \ 2+c-A>0$ and $T/R\in \mathcal{C}_{2+c-A}$ and $T^{-1}/R\in \mathcal{C}_A$. In view of Definition \ref{CAAAA}, this is equivalent to the existence of a Hilbert space $K'\supset H$ and unitaries $U_1, U_{-1}\in\mathcal{B}(K')$ such that 
\begin{equation}\label{actual1}
R^{-n}T^n=(2+c-A)^{1/2}P_H U^n_1(2+c-A)^{1/2}|_H, \hspace*{0.3 cm} \forall n\ge 1,   
\end{equation}
and 
\begin{equation}\label{actual2}
R^{-n} T^{-n}=A^{1/2}P_H U^n_{-1} A^{1/2}|_H, \hspace*{0.3 cm} \forall n\ge 1.  
\end{equation}
Thus, the content of Theorem \ref{2002} is that \ref{naivedil} holds if and only if there exists $A\in\mathcal{B}(H)$ with $0<A<2+c$ such that (\ref{actual1}) and (\ref{actual2}) hold. It would be of interest to find a direct proof of this assertion, using only the dilations $U_1, U_{-1}$ and $N.$

 \begin{proof}[Proof of Theorem \ref{2002}]
All that is left is to establish the converse of Lemma \ref{onedir}. 
\par 
Accordingly, assume $T\in \mathbb{CDLA}_R(c)$, with $\sigma(T)\subset A_R$. By Arveson's Theorem, $S_{R, c}$ extends to a completely positive map $\Psi_{R, c}: C(\partial A_R)\to\mathcal{B}(H)$. Next, consider the mapping 
$$\psi_{R, c}: C(\partial A_R)\to\mathcal{B}(H) $$
 $$f\mapsto \frac{1}{2+c}\bigg[ \int_{\partial A_R}\mu(\sigma, T)f(\sigma)ds +\frac{c}{2\pi i}\int_{C_1}\frac{f}{\sigma}\sigma'ds\bigg],$$
 which is an alternate (not necessarily positive) extension of $S_{R, c}$. Since $||\mu(\sigma, T)||$ is uniformly bounded with respect to $\sigma$ (because of the assumption $\sigma(T)\subset A_R$), we can estimate
 $$\bigg| \bigg|\int_{\partial A_R} \mu(\sigma, T)f(\sigma)ds\bigg|\bigg|\le ||f||_{\infty}\int_{\partial A_R} ||\mu(\sigma, T)||ds \le M||f||_{\infty}, \hspace{0.3 cm} \forall f\in C(\partial A_R).$$
 Thus, $\psi_{R, c}$ is bounded. Observe also that both $\psi_{R,c}$ and $\Psi_{R,c}$ are actually extensions of the (completely positive) map $\tilde{S}_{R, c}: \mathcal{M}_R\to\mathbb{C}$ defined in the proof of Theorem \ref{doublelayer}. For $\psi_{R,c}$ this is obvious, while for $\Psi_{R,c}$ it holds because the completely contractive map $S_{R, c}$ has a unique completely positive extension to the closure of $\mathcal{A}(A_R)+\mathcal{A}(A_R)^*$.\par  
Now, define $S: C(\partial A_R)\to\mathcal{B}(H)$ as $S=\psi_{R, c}-\Psi_{R, c}.$ Fix an orthonormal basis $\{e_j\}$ of $H$ and put 
$$S_{ij}:C(\partial A_R)\to\mathbb{C}$$
$$f\mapsto \langle S(f) e_j, e_i\rangle.$$
$S_{ij}$ will then be a bounded linear functional that vanishes on $\mathcal{M}_R$, which is a codimension one subspace of $C(\partial A_R).$ Hence, each $S_{ij}$ lies in the one-dimensional annihilator $\text{Ann}[\mathcal{M}_R]\subset \{L: C(\partial A_R)\to \mathbb{C} \text{ linear, bounded}\}$ of $\mathcal{M}_R$. But we also know that the (nonzero) map $$\phi: C(\partial A_R)\to\mathbb{C}$$
$$f\mapsto \frac{1}{(2+c)}\frac{1}{2\pi i}\int_{\partial A_R}\frac{f(\zeta)}{\zeta}d\zeta$$
lies in $\text{Ann}[\mathcal{M}_R]$. Thus, for every $i, j$, there exists $k_{ij}\in\mathbb{C}$  such that $S_{ij}=k_{ij}\phi $. Define the (a priori unbounded) operator $A$ acting on $H$ by $\langle A e_j, e_i\rangle=k_{ij}$, for all $i, j,$ and let $A_J$ denote its compression to $H_J:=\text{span}\{e_j : j\in J\}$, where $J$ is any finite subset of $\mathbb{N}.$  Hence, we obtain 
\begin{equation}\label{defineA}
    \frac{1}{2+c}\bigg[\frac{1}{2\pi i}\int_{\partial A_R}\frac{f(\zeta)}{\zeta}d\zeta\bigg] A_J=P_{H_J} S(f)|_{H_J},
\end{equation}
for every $f\in C(\partial A_R)$ and every $J.$ Set $f=f_0$ in this last equality, where $f_0\equiv 1$ on $C_1$ and $f_0\equiv 0$ on $C_{-1}$. This gives us 
$$A_J=(2+c)P_{H_J}S(f_0)|_{H_J},$$
for all $J.$ Since $S(f_0)=\psi(f_0)-\Psi(f_0)$ is bounded and self-adjoint, we obtain that $A$ is a bounded, self-adjoint operator. Also, in view of (\ref{defineA}), we may deduce that 
$$\Psi_{R, c}(f)=\psi_{R, c}(f)-S(f)=\int_{\partial A_R}\nu_{A, c}(\sigma, T)f(\sigma)ds$$
for every continuous $f$, where 
$$\nu_{A, c}(\sigma, T)=\begin{cases} \mu(\sigma, T)+\frac{1}{2\pi i}\frac{\sigma'}{\sigma}(c-A), \hspace{0.3 cm} \text{ for }\sigma\in C_1, \\
\mu(\sigma, T)-\frac{1}{2\pi i}\frac{\sigma'}{\sigma}(A), \hspace{0.3 cm}\text{ for }\sigma\in  C_{-1}.
\end{cases}$$
Since $\Psi_{R, c}$ is (completely) positive on $C(\partial A_R)$, we easily obtain that \\ $\nu_{A, c}(\sigma, T)\ge 0$ for every $\sigma\in \partial A_R,$ hence (as in the proof of Lemma \ref{onedir}) $T/R\in\mathcal{C}_{2+c-A}$ and $T^{-1}/R\in\mathcal{C}_{A}$. The fact that $\mathcal{C}_{2+c-A}$ and $\mathcal{C}_{A}$ are non-empty immediately implies  $0\le A\le 2+c$ (see the remark after Lemma \ref{CAdescr}). But we also know that $T$ is invertible, so Lemma \ref{notinvertCA} tells us that both $A$ and $2+c-A$ have to be invertible as well. This concludes the proof. \end{proof}

We now drop the assumption $\sigma(T)\subset A_R$.

\begin{theorem}\label{droptheorem}
Assume $ R>1, c>-2$ and $T\in\mathcal{B}(H)$ satisfies $\sigma(T)\subset \overline{A_R}.$ Then, there exists $A\in\mathcal{B}(H)$ such that $A>0, \ 2+c-A>0$ and $T/R\in \mathcal{C}_{2+c-A}$ and $T^{-1}/R\in \mathcal{C}_A $ if and only if $T\in \mathbb{CDLA}_{R'}(c)$ for every $R'>R.$
\end{theorem}

\begin{proof}
First, assume  $T\in \mathbb{CDLA}_{R'}(c)$ for every $R'>R$ and fix a decreasing sequence $\epsilon_k\to 0.$ Put $R_k=R+\epsilon_k$. Since $\sigma(T)\subset A_{R_k},$ Theorem \ref{2002} tells us that there exists
 a sequence $\{A_k\}\subset \mathcal{B}(H)$ such that $0<A_k<2+c$ and $T/R_k\in\mathcal{C}_{2+c-A_k}$ and $T^{-1}/R_k\in\mathcal{C}_{A_k},$ for all $k\ge 1.$ Now, from $T/R_k\in\mathcal{C}_{2+c-A_k}$ we get
\begin{equation}\label{11112}
\Big\langle \big(2\Re(1-zT/R_k)^{-1}+c-A_k\big)v, v\Big\rangle \ge 0,
\end{equation}
for all $z\in\mathbb{D}, v\in H$ and $k\ge 1$. Since $\{A_k\}$ is uniformly bounded, we may replace it, without loss of generality, by a WOT-convergent subsequence. Note also that $\sigma(zT/R_k)\subset \mathbb{D},$ for all $z, k.$ Letting $k\to\infty$ in (\ref{11112}) (while keeping $z$ and $v$ fixed) is then easily seen to imply
$$\Big\langle \big(2\Re(1-zT/R)^{-1}+c-A\big)v, v\Big\rangle \ge 0, $$
where $A$ is the WOT limit of $\{A_k\}$ (notice that $A$ has to be self-adjoint, being the WOT limit of self-adjoint operators). Since this last inequality holds for any $z\in\mathbb{D}$ and $v\in H,$ we conclude that $T/R\in\mathcal{C}_{2+c-A}$, while an entirely analogous argument shows that $T^{-1}/R\in\mathcal{C}_A.$ Finally, the fact that both $\mathcal{C}_{2+c-A}$ and $\mathcal{C}_{A}$ contain an invertible operator implies, as seen previously, that $0<A<2+c$. \par
For the converse, observe that (in view of Lemma \ref{CAdescr}) having $T/R\in \mathcal{C}_{2+c-A}$ and $T^{-1}/R\in \mathcal{C}_A $ implies that $T/{R'}\in \mathcal{C}_{2+c-A}$ and $T^{-1}/{R'}\in \mathcal{C}_A $ for every $R'>R.$ Since $\sigma(T)\subset A_{R'}$, Theorem \ref{2002} allows us to deduce that $T\in \mathbb{CDLA}_{R'}(c)$ for every $R'>R.$
\end{proof}

Now, we record the following analogue of Proposition \ref{DLAstuff} for $\mathbb{CDLA}_R(c)$. The proof is essentially an application of Proposition \ref{DLAstuff} combined with Lemma \ref{onedir}, so we omit the details.

\begin{proposition}\label{CDLAstuff}
Let $c\in\mathbb{R}$, $R>1$ and assume $N\in\mathcal{B}(H)$ is normal.
\begin{itemize}
    \item[(i)] If $c>-2,$ then $\mathbb{CDLA}_R(c)$ will be non-empty if and only if there exists $s\in (0, 2+c)$ such that $\mathcal{C}_{2+c-s, s}(R)$ is non-empty. \\
    In particular, if $c\ge 0,$ then $\mathbb{CDLA}_R(c)$  will be non-empty for every $R>1.$
    \item[(ii)] If $c\le -2,$ then $\mathbb{CDLA}_R(c)$ will be empty for every $R>1.$
    \item[(iii)] $\mathbb{CDLA}_R(c)\subset \mathcal{C}_{2-c, 2-c}(R),$ for every $c>-2. $
    \item[(iv)] If, in addition, we assume $c\ge 0,$ then $N\in \mathbb{CDLA}_R(c)$ is equivalent to $R^{-2}\le N^*N\le R^2.$
\end{itemize}
\end{proposition}

It is also worth noting that, like $\mathbb{DLA}_R(c),$ $\mathbb{CDLA}_R(c),$ is (eventually) strictly monotone with respect to $c.$ This has essentially already been given to us by the proof of Theorem \ref{monot} plus Lemma \ref{onedir}.
\begin{theorem}\label{cmonot}
If $-2<c<c'$ and $\mathbb{CDLA}_R(c')$ is non-empty, then $$\mathbb{CDLA}_R(c)\subsetneq \mathbb{CDLA}_R(c').$$
\end{theorem}
\begin{proof} If $c<c'$, it is obvious by the definition of $\mathbb{CDLA}_R(c)$ that $\mathbb{CDLA}_R(c)\subseteq\mathbb{CDLA}_R(c').$ To show that the inclusion is actually strict, we  divide the proof into two cases, like with Theorem \ref{monot}. \par 
First, assume that there exist $s, t>0$ such that $s\ge 1, s+t=2+c'$ and $\mathcal{C}_{s. t}(R)$ is non-empty. As in the proof of Theorem \ref{monot}, we may assume that $t<1$ and $\frac{2}{R^2+1}<t$. In this setting, we were able to construct $T\in \mathcal{C}_{s, t}(R)\subset \mathbb{CDLA}_R(c')$ such that $T\notin \mathbb{DLA}_R(c),$ hence also 
$T\notin \mathbb{CDLA}_R(c).$ We thus obtain strict inclusion.\par 
 On the other hand, assume that there do not exist $s, t>0$ such that $s\ge 1, s+t=2+c'$ and $\mathcal{C}_{s. t}(R)$ is non-empty. In this setting, we found $\delta>0$ and $a\in\mathbb{C}$ such that $a\in\mathcal{C}_{\delta, 2+c'-\delta}(R)\subset \mathbb{CDLA}_R(c')$, but $a\notin \mathbb{CDLA}_R(c),$ as desired.
\end{proof}
We end with a question. While the inclusion $ \mathbb{CDLA}_R(c)\subseteq  \mathbb{DLA}_R(c)$ is obvious, we have not been able to determine whether it is actually strict or not.
\begin{question}\label{quest}
Let $R>1, c>-2$ and assume $\mathbb{CDLA}_R(c)$ is non-empty. Is it true that 
$$\mathbb{CDLA}_R(c)\subsetneq \mathbb{DLA}_R(c) ?$$
\end{question}
A negative answer to the above question would imply that, given $\sigma(T)\subset A_R$, having 
 $$2\Re\big[(1-zT/R)^{-1}+(1-wT^{-1}/R)^{-1}\big]-2+c\ge 0, \hspace{0.4 cm} \forall z, w\in\mathbb{D},$$
 is equivalent to the existence of $0<A<2+c$ such that
 $$2\Re(1-zT/R)^{-1}+c-A\ge 0, \hspace{0.4 cm}\forall z\in\mathbb{D},$$
 and 
$$2\Re(1-wT^{-1}/R)^{-1}+A-2\ge 0, \hspace{0.4 cm}\forall w\in\mathbb{D}.$$
While this seems unlikely to hold, the computational difficulty in verifying membership conditions of the form $T/R\in\mathcal{C}_{2+c-A}(R)$ 
and $T^{-1}/R\in\mathcal{C}_{A}(R)$, for arbitrary $0<A<2+c$ and $T$ non-normal, does not make it easy to come up with a counterexample.

\small

\section{$K$-SPECTRAL ESTIMATES}\label{kspectral}
 \large 
\subsection{General estimates} Recall that, given a compact set $X\subset\mathbb{C}$ and $T\in\mathcal{B}(H)$ such that $\sigma(T)\subset X$, $X$ is said to be a $K$-spectral set for $T$ if the inequality 
$$||f(T)||\le K \sup_{z\in X}|f(z)|$$
holds for every rational
function $f$ with poles off of $X$. $X$ will be called a complete $K$-spectral set if the above inequality holds for all matrices with rational coefficients. \par 
In this short subsection, we show how the methods established in \cite{LiCaldwellGreen} and \cite{CrouzGreen}  can be used to derive $K$-spectral estimates for $\mathbb{DLA}_R(c)$ and $\mathbb{CDLA}_R(c)$. We shall need a few  preliminary lemmata, the scalar-valued versions of which are all contained in \cite{CrouzGreen}.
\begin{lemma} \label{ccCf}
 The map 
 $$\alpha: \mathcal{A}(A_R)\to \mathcal{A}(A_R)  $$
 $$f\mapsto C\overline{f}$$
is completely contractive, for every $R>1.$
\end{lemma}
 \begin{proof}
By \cite[Lemma 8]{CrouzGreen},  we know that $||\alpha(f)||\le ||f||$ whenever $f$ is a scalar-valued rational function in $A_R$ that is bounded by $1.$ A standard approximation argument  shows that $\alpha$ must be contractive. The proof of \cite[Lemma 2.1]{bivariate} then implies that $\alpha$ is completely contractive.
 \end{proof}

\begin{lemma}\label{RansSchw}
  Let $R>1$ and $T\in\mathcal{B}(H)$ be such that $\sigma(T)\subset A_R.$  Assume that there exists a bounded linear functional $\gamma: \mathcal{A}(A_R)\to\mathbb{C}$ and a constant $p>0$ such that the mapping 
  $$f\mapsto \frac{1}{2p}\big(f(T)+\alpha(f)(T)^*+\gamma(f)\big) $$
  is completely contractive on $\mathcal{A}(A_R)$. Then, $\overline{A_R}$ is a complete $K$-spectral set for $T$ with constant 
  $K=p+\sqrt{1+p^2+||\gamma||_{\text{cb}}}$.
\end{lemma}
\begin{proof}
In the setting of \cite[Theorem 2]{CrouzGreen}, replace $\Omega$ by $A_R,$ $c_1$ by $1$ (this is possible because of Lemma \ref{ccCf}), $c_2$ by $p$ and $\hat{\gamma}$ by $||\gamma||_{\text{cb}}$. While the proof of Theorem 2 in \cite{CrouzGreen} was given in the scalar-valued setting, it can be repeated, mutatis mutandis, in the matrix-valued setting to give the exact same $K$-spectral estimate (see also Remark (i) after the proof of Theorem 1.1 in \cite{RansfordOster}), where 
$$K=c_2+\sqrt{c_2^2+c_1+\hat{\gamma}}=p+\sqrt{1+p^2+||\gamma||_{\text{cb}}}.$$\end{proof}
We are now in a position to show:
\begin{theorem}\label{generalspectral}
  Let $c> -2$ and assume $T\in \mathbb{DLA}_{R}(c)$ (resp. $T\in \mathbb{CDLA}_{R}(c)$). Then, $\overline{A_R}$ will be a $K$-spectral (resp. complete $K$-spectral) set for $T$, where 
$$K=1+\frac{c}{2}+\sqrt{\Big(1+\frac{c}{2}\Big)^2+1+|c|}.$$
\end{theorem} 
\begin{proof} 
Let $T\in \mathbb{CDLA}_{R}(c)$ (the case $T\in \mathbb{DLA}_{R}(c)$ is essentially contained in \cite[Theorem 2]{CrouzGreen}). By assumption, the mapping 
$$f\mapsto \frac{1}{2+c}\big(f(T)+\alpha(f)(T)^*+ca_0\big)$$
is completely contractive on $\mathcal{A}(A_R)$. If $\gamma: \mathcal{A}(A_R)\to\mathbb{C}$ is given by $\gamma(\sum a_n z^n)=ca_0,$ it can be easily verified that $||\gamma||_{\text{cb}}=|c|$. Thus, one can apply Lemma \ref{RansSchw} with $p=1+c/2$ and $\gamma(f)=ca_0$ to deduce the desired result.  
\end{proof}

\begin{remark}\label{improvenum}
Let $\mathbb{NA}_R$ denote the \textit{numerical annulus}, i.e. the class of all $T\in\mathcal{B}(H)$ such that $w(T)\le R$ and $w(T^{-1})\le R.$ $K$-spectral estimates for $\mathbb{NA}_R$ have been studied in \cite{Annuluskspectral} and, more recently, in \cite{CrouzGreen}. Since $$\mathbb{NA}_R\equiv \mathcal{C}_{2, 2}(R)\subset\mathbb{DLA}_R(2),$$ Theorem \ref{generalspectral} tells us that $\overline{A_R}$ is a $(2+\sqrt{7})$-spectral set for $T$ whenever $T\in \mathbb{NA}_R$. This improves on the estimates from \cite[Section 6]{CrouzGreen}. 
\end{remark}
\begin{remark} \label{devriesremark} The main result of \cite{schwendevries} can be used to obtain sharper spectral estimates in certain cases. Indeed, let $T\in \mathbb{DLA}_{R}(c)$ be a matrix with $\sigma(T)\subset A_R$.  In the setting of \cite[Theorem 5]{schwendevries}, choose $A=\mathcal{A}(A_R)$ and set $\gamma(f)=f(T)$ and $\Phi(f)=C\overline{f}.$ Since $c\ge 0$, we have $\mathbb{QA}_R\subset\mathbb{DLA}_R(c)$, which implies that $||\gamma||\ge 2$ (see \cite{Tsiknote}). Now, assume, in addition, the existence of an \textit{extremal pair}  $(f_0, x_0)\in\mathcal{A}(A_R)\times H$ for $\gamma$ (see \cite[p. 2]{schwendevries}).   We have $||f_0||=||x_0||=1$ and $||\gamma||=||\gamma(f_0)x_0||$. Also, there exists an \textit{extremal measure} associated with $(f_0, x_0)$ (see 
\cite[Proposition 3]{schwendevries} and the discussion afterwards). This observation, combined with the contractivity of $\Phi$ (Lemma \ref{ccCf}), allows us to deduce that $|\langle \gamma(\Phi(f_0)f_0)x_0, x_0\rangle|\le 1$, see e.g. the proof of \cite[Theorem 11]{schwendevries}). Thus, if we define $\omega$ in the dual of $\mathcal{A}(A_R)$ as $\omega(\sum_n a_n z^n)=-ca_0$, \cite[Theorem 5]{schwendevries} implies that 
$$||\gamma||\le \frac{1}{2}||\gamma_{\Phi}-\omega||+\sqrt{\bigg(\frac{1}{2}||\gamma_{\Phi}-\omega||\bigg)^2+|\langle \gamma(\Phi(f_0)f_0)x_0, x_0\rangle|} $$
$$\le \frac{2+c}{2} ||S_{R, c}||+\sqrt{\bigg(\frac{2+c}{2} ||S_{R, c}||\bigg)^2+1} $$
$$\le 1+\frac{c}{2}+\sqrt{\bigg(1+\frac{c}{2}\bigg)^2+1},$$
which gives us the sharper constant $K'=1+\frac{c}{2}+\sqrt{\big(1+\frac{c}{2}\big)^2+1}.$ Note that if $T$ has distinct eigenvalues, then the existence of an extremal function $f_0$ can be obtained as in the proof of \cite[Theorem 2.1]{Boundsforanalytical} (see \cite[Section 3]{Annulusquotients} for the structure of solutions to extremal Pick problems over the annulus).
\end{remark}

\subsection{$2\times 2$ matrices} 
We now show Theorem \ref{2x2} from the introduction, which offers improved $K$-spectral estimates for $2\times 2$ matrices with a single eigenvalue. The key ingredient of the proof will be a function-theoretic result from \cite{Mccshen}. To more easily connect with the setting of that paper,  we will work
with the annulus,
$$\mathscr{A}_q:=\{q<|z|<1\},$$
which is conformally equivalent to $A_{1/\sqrt{q}}.$
The definition of $\mathbb{DLA}_R(c)$ can then be updated as follows:
\begin{definition}\label{newdef}
 Let $c\in\mathbb{R}$ and $0<q<1$. $\mathscr{DLA}_q(c)$ denotes the class of all operators $T\in\mathcal{B}(H)$ such that 
\begin{itemize}
    \item[(i)] $\sigma(T)\subset\overline{\mathscr{A}_q}$ and 
    \item[(ii)] $2\Re\big[(1-zT)^{-1}+(1-wqT^{-1})^{-1}\big]-2+c\ge 0, \hspace{0.4 cm} \forall z, w\in\mathbb{D}.$
\end{itemize}   
\end{definition}
Now, for $w\in\mathbb{D}$ and $a\in \mathscr{A}_q$, define $\psi_w(z)=\frac{z-w}{1-\overline{w}z}$ and 
$$\mathcal{F}_{a, w}=\{f:\overline{\mathscr{A}_q}\to\mathbb{D}\text{ }|\text{ } f \text{ analytic } \text{ and }  f(a)=w\}.$$

\begin{lemma}\label{basicestimate}
Let  $w\in\mathbb{D}$ and $a\in \mathscr{A}_q$. Then, 
$$\sup\{|f'(a)|\text{ }|\text{ }f\in\mathcal{F}_{a, w}\}\le (1-|w|^2)\bigg(\frac{1}{1-|a|^2}+\frac{q}{|a|^2-q^2} \bigg).$$
\end{lemma}
\begin{proof}
Put $k^q(a, a)=\sum_{n\in\mathbb{Z}}\frac{|a|^{2n}}{1+q^{2n+1}}.$ The solution to the above extremal problem for $w=0$ can be found in \cite[p. 1119]{Mccshen}. In particular, it is known that
$$\sup\{|f'(a)|\text{ }|\text{ }f\in\mathcal{F}_{a, 0}\}=k^q(a, a).$$
Now, if $h\in\mathcal{F}_{a, w},$ it can be easily verified that $\psi_w\circ h\in \mathcal{F}_{a, 0},$ hence 
$$\frac{|h'(a)|}{1-|w|^2}=|(\psi_w\circ h)'(a)|\le k^q(a, a).$$
Since $h\in\mathcal{F}_{a, w}$ was arbitrary, we can deduce that 
$$\sup\{|f'(a)|\text{ }|\text{ }f\in\mathcal{F}_{a, w}\}\le (1-|w|^2)k^q(a, a) $$
$$=(1-|w|^2)\bigg(\sum_{n=0}^{\infty}\frac{|a|^{2n}}{1+q^{2n+1}}+\sum_{n=-1}^{-\infty}\frac{|a|^{2n}}{1+q^{2n+1}}    \bigg) $$
$$\le (1-|w|^2)\bigg(\sum_{n=0}^{\infty}|a|^{2n}+\frac{1}{q}\sum_{n=1}^{\infty}\frac{q^{2n}}{|a|^{2n}} \bigg) $$
$$=(1-|w|^2)\bigg(\frac{1}{1-|a|^2}+\frac{q}{|a|^2-q^2} \bigg).$$
\end{proof}
We also require the following computational lemmata.
\begin{lemma}\label{Membership}
Let $a, u\in\mathbb{C}$ and assume $T=\begin{pmatrix} a & u \\ 0 & a \end{pmatrix}\in \mathscr{DLA}_{q}(c)$ with $q<|a|<1$. Then,

$$\bigg|\frac{e^{i\theta}}{(1-ae^{i\theta})^2} -\frac{qe^{i\psi}}{a^2(1-qe^{i\psi}a^{-1})^2}  \bigg||u| \le 2\Re\bigg(\frac{1}{1-ae^{i\theta}}+\frac{1}{1-qe^{i\psi}a^{-1}} \bigg)+c-2,$$
for all $\theta$ and $\psi$.
\end{lemma}
\begin{proof}
 The proof of Lemma \ref{DLAfromztotheta} carries over to the $\mathscr{DLA}_q(c)$ setting. Thus, the fact that $T\in \mathscr{DLA}_{q}(c)$ and $\sigma(T)=\{a\}\subset \mathscr{A}_q$ allows us to deduce 
  $$2\Re\big[(1-e^{i\theta}T)^{-1}+(1-e^{i\psi}qT^{-1})^{-1}\big]-2+c\ge 0, \hspace{0.4 cm} \forall \theta, \psi\in [0, 2\pi).$$
  Taking determinants then leads to the desired inequality.
\end{proof}

\begin{lemma} \label{matrixnorm}
 Let $C>0$ and $w\in\mathbb{D}.$ Then, 
 $$\bigg|\bigg|\begin{pmatrix}
 w & C(1-|w|^2) \\
 0 & w
 \end{pmatrix}
 \bigg|\bigg|\le \max\{1, C\}.$$
\end{lemma}
\begin{proof}
Given any matrix of the form $P=\begin{pmatrix}
 e & f \\
 0 & e
 \end{pmatrix}$, it is well-known that $||P||\le 1$ if and only if $|f|\le 1-|e|^2.$ Thus, if $C\le 1,$ we immediately obtain that 
$$\bigg|\bigg|\begin{pmatrix}
 w & C(1-|w|^2) \\
 0 & w
 \end{pmatrix}
 \bigg|\bigg|\le 1.$$
Now, assume $C>1$. Note that
$$C^2-\begin{pmatrix}
 w & C(1-|w|^2) \\
 0 & w
 \end{pmatrix}\begin{pmatrix}
 w & C(1-|w|^2) \\
 0 & w
 \end{pmatrix}^*$$ $$=\begin{pmatrix}
 -|w|^2+C^2(2-|w|^2)|w|^2 & -C\overline{w}(1-|w|^2) \\
 -Cw(1-|w|^2) & C^2-|w|^2
 \end{pmatrix}\ge 0,$$
 since the $(1, 1)$-entry of this last matrix is clearly positive, while its determinant is equal to 
 $$|w|^2(C^2-|w|^2)(2C^2-C^2|w|^2-1)-C^2|w|^2(1-|w|^2)^2 $$
 $$=|w|^2(C^2-1)(2C^2-|w|^2(1+C^2))>0.$$
 This concludes the proof.
\end{proof}
We are now prepared for the main result of this subsection.
\begin{proof}[Proof of Theorem \ref{2x2}]
First, we will prove the theorem in the setting of $\mathscr{A}_q.$ \par 
  Let $T\in \mathscr{DLA}_{q}(c)$ be a $2\times 2$ matrix with a single eigenvalue. Since unitary equivalence respects $K$-spectral estimates (see e.g. \cite[Example 4, p. 107-5]{badeaspectral}), we may assume that $T$ is of the form $\begin{pmatrix} a & u \\ 0 & a \end{pmatrix}$. We may also take $a>0$ (as $\mathscr{A}_q$ is invariant under rotations).  Finally, it suffices to work with $q<a<1$ (the general case follows by a standard approximation argument). \par 
Now, let $f: \overline{\mathscr{A}_q}\to\mathbb{D}$ be analytic. We may write
$$||f(T)||=\bigg|\bigg| \begin{pmatrix} f(a) & f'(a)u \\ 0 & f(a) \end{pmatrix}\bigg|\bigg|$$ 
$$=\bigg|\bigg| \begin{pmatrix} f(a) & |f'(a)u| \\ 0 & f(a) \end{pmatrix}\bigg|\bigg| $$
$$=\bigg|\bigg| \begin{pmatrix} f(a) & \frac{|f'(a)u|}{1-|f(a)|^2}(1-|f(a)|^2) \\ 0 & f(a) \end{pmatrix}\bigg|\bigg|.$$
If $\frac{|f'(a)u|}{1-|f(a)|^2}\le 1,$ Lemma \ref{matrixnorm} gives us $||f(T)||\le 1,$ which is stronger than the desired estimate. On the other hand, if  $\frac{|f'(a)u|}{1-|f(a)|^2}> 1,$  Lemmata \ref{basicestimate} and \ref{matrixnorm} imply that 
$$||f(T)||=\bigg|\bigg| \begin{pmatrix} f(a) & \frac{|f'(a)u|}{1-|f(a)|^2}(1-|f(a)|^2) \\ 0 & f(a) \end{pmatrix}\bigg|\bigg|$$
$$\le \frac{|f'(a)|}{1-|f(a)|^2}|u|$$ $$\le \bigg(\frac{1}{1-a^2}+\frac{q}{a^2-q^2}\bigg)|u|.$$
Assume that $a\ge \sqrt{q}$. In Lemma \ref{Membership}, take $\theta=0, \psi=\pi$. The resulting bound on $|u|$ then allows us to write:

$$||f(T)|| $$ $$\le  \bigg(\frac{1}{1-a^2}+\frac{q}{a^2-q^2}\bigg)\bigg(\frac{1}{(1-a)^2}+\frac{q}{a^2(1+q/a)^2} \bigg)^{-1}\bigg(\frac{2}{1-a}+\frac{2}{1+q/a}+c-2 \bigg)$$ $$\le 2+c\frac{1-q}{1+q},$$
where the last inequality can be seen (after some computations) to be equivalent to $a\ge \sqrt{q}$. This concludes the proof in this case. If $a<\sqrt{q}$, one can choose $\theta=\pi$ and $\psi=0$ in the above calculation and argue in an analogous manner.  
Thus, we have shown that $\overline{\mathscr{A}_q}$ is a $K_q$-spectral for $T\in\mathscr{DLA}_q(c)$
 whenever $T$ is a $2\times 2$ matrix with a single eigenvalue, where $K_q=2+c\frac{1-q}{1+q}.$ \par 
 We now convert this estimate to the $A_R$-setting. Assume $T\in\mathbb{DLA}_R(c)$ is a $2\times 2$ matrix with a single eigenvalue and set $q=R^{-2}.$ Then, $\tilde{T}:=T/R\in \mathscr{DLA}_q(c)$ and is also, evidently, still a $2\times 2$ matrix with a single eigenvalue. In view of our previous result, $\overline{\mathscr{A}_q}$ will be a $K$-spectral set for $\tilde{T}$, where 
 $$K=2+c\frac{1-q}{1+q}=2+c\frac{R^2-1}{R^2+1}.$$
 This is easily seen to imply (see e.g. \cite[Fact 2, p. 107-3]{badeaspectral}) that $\overline{A_R}$ is a $K$-spectral set for $T$, which concludes our proof. 
\end{proof}

 \par\textit{Funding}. Jury was partially supported by National Science Foundation Grant DMS 2154494. Tsikalas was partially supported by National Science Foundation Grant DMS 2054199 and by Onassis Foundation - Scholarship ID: F ZR 061-1/2022-2023.  \\ 

 \par\textit{Acknowledgements}. The research was conducted while the second author was visiting the Department of Mathematics at the University of Florida in Gainesville. He gratefully acknowledges the hospitality of his gracious hosts, Michael Jury and Scott McCullough.

\printbibliography

\end{document}